\documentclass[11pt]{amsart}
\usepackage{amscd,amsmath,amssymb,amsfonts}
\usepackage[cmtip,all]{xy}
\usepackage{hyperref}
\usepackage{stmaryrd}
\usepackage{tabularx}

\usepackage[mathscr]{euscript}

\newcounter{zaehler}
\theoremstyle{plain}

\newtheorem{introthm}[zaehler]{Theorem}
\newtheorem{introcor}[zaehler]{Corollary}
\newtheorem{thm}{Theorem}[section]

\newtheorem{lem}[thm]{Lemma}
\newtheorem{cor}[thm]{Corollary}
\newtheorem{prop}[thm]{Proposition}
\newtheorem{const}[thm]{Construction}

\theoremstyle{definition}
\newtheorem{defn}[thm]{Definition}

\newtheorem{claim}[thm]{Claim}
\newtheorem{rmk}[thm]{Remark}

\newcommand{\eq}[2]{\begin{equation}\label{#1}#2 \end{equation}}

\newcommand{\Hom}{{\rm Hom}}
\newcommand{\im}{{\rm im}}

\newcommand{\sA}{{\mathcal A}}
\newcommand{\sB}{{\mathcal B}}
\newcommand{\sC}{{\mathcal C}}
\newcommand{\sD}{{\mathcal D}}

\newcommand{\sF}{{\mathcal F}}

\newcommand{\sO}{{\mathcal O}}

\newcommand{\sU}{{\mathcal U}}
\newcommand{\sV}{{\mathcal V}}

\newcommand{\derived}[1]{\mathscr{#1}}

\newcommand{\dX}{\derived{X}}
\newcommand{\dY}{\derived{Y}}
\newcommand{\dZ}{\derived{Z}}
\newcommand{\dD}{\derived{D}}
\newcommand{\dE}{\derived{E}}
\newcommand{\dV}{\derived{V}}
\newcommand{\dho}{{h}}
\newcommand{\dU}{\derived{U}}
\newcommand{\dW}{\derived{W}}
\newcommand{\Bld}{\derived{B}\mathrm{l}}

\newcommand{\QCoh}{\mathrm{QCoh}}
\newcommand{\Perf}{\mathrm{Perf}}

\newcommand{\Sch}{\mathrm{Sch}}

\newcommand{\Ring}{\mathrm{Ring}}
\newcommand{\cRing}{\mathrm{sRing}}
\newcommand{\Cat}{\mathrm{Cat}}

\newcommand{\cat}[1]{\mathbf{#1}}
\newcommand{\sRing}{\cat{sRing}}
\newcommand{\dSch}{\cat{dSch}}
\newcommand{\dAff}{\cat{dAff}}
\newcommand{\dMod}{\mathrm{Mod}}

\newcommand{\A}{{\mathbb A}}

\newcommand{\N}{{\mathbb N}}
\renewcommand{\P}{{\mathbb P}}

\newcommand{\Z}{{\mathbb Z}}

\newcommand{\id}{{\rm id\hspace{.1ex}}}
\DeclareMathOperator{\Ind}{Ind}
\newcommand{\Sp}{{\rm Sp}}

\newcommand{\prolim}[1]{\operatorname{``}\lim\limits_{#1}   \operatorname{''}}

\DeclareMathOperator{\Proj}{Proj}
\DeclareMathOperator{\Tor}{Tor}

\newcommand{\GL}{{\rm GL}}

\newcommand{\Bl}{{\rm Bl}}

\DeclareMathOperator{\coker}{coker}
\DeclareMathOperator*{\colim}{colim}

\DeclareMathOperator{\Spec}{Spec}
\DeclareMathOperator{\Ho}{Ho}
\DeclareMathOperator{\Mod}{Mod}
\newcommand{\ba}{\mathbf{a}}

\newcommand{\Ko}{\mathrm{Ko}}
\newcommand{\GLh}{\widehat{\mathrm{GL}}}
\newcommand{\BGL}{\mathrm{BGL}}
\newcommand{\BGLh}{\mathrm{B\widehat{GL}}}
\newcommand{\rB}{\mathrm{B}}
\newcommand{\hofib}{\mathrm{hofib}}
\newcommand{\rM}{\mathrm{M}}
\newcommand{\rE}{\mathrm{E}}
\newcommand{\rT}{\mathrm{T}}
\newcommand{\rG}{\mathrm{G}}
\newcommand{\cdh}{\mathrm{cdh}}
\newcommand{\ra}{\mathrm{a}}
\newcommand{\rL}{\mathrm{L}}
\newcommand{\Nil}{\mathrm{Nil}}
\newcommand{\Coh}{\mathrm{Coh}}
\newcommand{\Vect}{\mathrm{Vec}}
\newcommand{\frY}{\mathfrak{Y}}
\newcommand{\frU}{\mathfrak{U}}

\numberwithin{equation}{section}

\begin{document}
\date{\today}

\title{Algebraic \texorpdfstring{$K$}{K}-theory and descent for blow-ups} 

\author{Moritz Kerz}
\address{Fakult\"at f\"ur Mathematik, Universit\"at Regensburg, 93040 Regensburg, Germany}
\email{moritz.kerz@ur.de}
\author{Florian Strunk}
\email{florian.strunk@ur.de}
\author{Georg Tamme}
\email{georg.tamme@ur.de}

\thanks{The authors are supported by the DFG through CRC 1085 \emph{Higher Invariants} (Universit\"at Regensburg)}

\begin{abstract}
We prove that algebraic $K$-theory satisfies `pro-descent' for abstract blow-up squares of noetherian schemes.
As an application we derive Weibel's conjecture on the vanishing of negative $K$-groups.
\end{abstract}

\maketitle

\setcounter{tocdepth}{2}

\section{Introduction}
\label{sec:Introduction}

Let $X$ be a noetherian scheme. A cartesian diagram of schemes
\begin{equation}\label{intro.eq1}
\begin{split}
\xymatrix{
\tilde X \ar[d] & E \ar[l] \ar[d] \\
X & \ar[l] Y
}
\end{split}
\end{equation}
is called an \emph{abstract blow-up square} if $\tilde X \to X$ is proper, $Y \to X$ is a
closed immersion, and the induced morphism $\tilde X \setminus E \to X \setminus Y$ is an
isomorphism.
It is an interesting question to ask whether algebraic $K$-theory satisfies descent for
the abstract blow-up square~\eqref{intro.eq1}, i.e.~whether there exists a long exact
sequence of algebraic $K$-groups
\begin{equation}\label{intro:mayervietoris}
\cdots \to K_i(X) \to K_i(Y) \oplus K_i(\tilde X) \to  K_i(E) \to K_{i-1}(X) \to \cdots .
\end{equation}

Historically, one of the first results in this direction states that for $X$ affine and $\tilde X\to X$ finite there is an exact sequence
\begin{equation}\label{intro.eq11}
 K_1(E) \to K_{0}(X) \to K_0(Y) \oplus K_0(\tilde X)  \to K_0(E) \to K_{-1}(X)\to \cdots .
\end{equation}
This is a consequence of the classical excision theorem of Bass
\cite[Thm.~XII.8.3]{Bass68} and of the Artin--Rees lemma, see the proof of Proposition~\ref{pr.findesc}. 

Bass \cite[Thm.~XII.10.4]{Bass68} used the exact sequence~\eqref{intro.eq11} to calculate the negative $K$-groups of an
 affine, one-dimensional, Nagata scheme $X$. Indeed, applying for such $X$ the exact
sequence~\eqref{intro.eq11} for the normalization $\tilde X$ of
$X_{\rm red}$ and a suitable $0$-dimensional subscheme $Y$ of $X_{\rm red}$ one obtains that $K_{i}(X)=0$ for $i<-1$ and that $K_{-1}(X)$ is a finitely
generated free
abelian group depending only on ``combinatorial'' data associated with $X$. For  higher
dimensional analogs of these two observations see Theorem~\ref{intro.thm.wc} and Corollary~\ref{intro.cor} below.

Another important descent result for $K$-theory of blow-up squares is due to Thomason~\cite{Thomason-blowup}. He shows that
$K$-theory satisfies descent for squares \eqref{intro.eq1} in which $Y\to X$ is a
regular closed immersion and  $\tilde X$  is the blow-up of $X$ in $Y$.
However, in general descent for $K$-theory of an abstract
blow-up square does not hold. 

In view of the following  observation of Grothendieck about coherent
sheaf cohomology it  is natural to ask for a pro-version of the descent
sequence \eqref{intro:mayervietoris} as formulated in our main Theorem~\ref{intro.main.thm} below.
Denote by $Y_{n}$ respectively $E_{n}$ the $n$-th infinitesimal thickening of $Y$ in
$X$ respectively $E$ in $\tilde X$. 
In the setting of the abstract blow-up square~\eqref{intro.eq1} Grothendieck's theorem on formal functions
\cite[III.4.1]{EGA} says that for any coherent sheaf $\sF$ on $X$ the square
\begin{equation}\label{into.eq1}
\begin{split}
\xymatrix{
  R\Gamma (X, \sF_{X}) \ar[r] \ar[d] & \prolim{n} R\Gamma ( Y_n, \sF_{ Y_n}) \ar[d]  \\
R\Gamma (\tilde X, \sF_{\tilde X}) \ar[r]  & \prolim{n} R\Gamma ( E_n, \sF_{ E_n})  
}
\end{split}
\end{equation}
is homotopy cartesian in the sense of pro-systems, 
i.e.\ it induces a long exact sequence of cohomology pro-groups. Here for a scheme $T$ over $X$ we write $\sF_T$ for the pullback of $\sF$ to $T$.

Our main result says that the analog of the square~\eqref{into.eq1} for
algebraic \mbox{$K$-theory} is homotopy cartesian.
\begin{introthm}\label{intro.main.thm}
For any abstract blow-up square~\eqref{intro.eq1} the diagram of pro-spectra of non-connective algebraic $K$-theory
\begin{equation}\label{into.eq2}
\begin{split}
\xymatrix{
K(X) \ar[r] \ar[d] & \prolim{n} K(Y_{n}) \ar[d] \\
K(\tilde X) \ar[r] & \prolim{n} K(E_{n})
}
\end{split}
\end{equation}
is  homotopy cartesian.
\end{introthm}

The notion of  homotopy cartesian square of pro-spectra we use, see
Definition~\ref{ex.defn.prospec}, is slightly weaker than the one resulting from the model
category structure of \cite{Isaksen04}. However, the squares~\eqref{into.eq1}
and~\eqref{into.eq2} are homotopy cartesian for both notions, see Remark~\ref{ex.rmkcohoca}.

More concretely, Theorem~\ref{intro.main.thm} is equivalent to one of the following two equivalent
statements.
\begin{itemize}
\item[(i)]
The canonical morphism of pro-systems of relative $K$-groups
\[
\prolim{n} K_i(X,Y_n) \to \prolim{n} K_i(\tilde X,E_n)
\]
is an isomorphism for all $i\in \Z$.
\item[(ii)]
There exists a natural long exact
sequence of pro-groups
\begin{multline}\label{into.eq3}
\cdots \to K_i(X) \to \prolim{n}  K_i(Y_{n}) \oplus K_i(\tilde X) \to\\\to   \prolim{n}
K_i(E_{n}) \to K_{i-1}(X) \to \cdots .
\end{multline}
\end{itemize}
In particular, Theorem~\ref{intro.main.thm} can be seen as a generalization of Bass' exact
sequence~\eqref{intro.eq11} for a finite abstract blow-up, since for affine schemes the
non-positive $K$-groups  only depend on the underlying reduced scheme.

Morrow has shown Theorem~\ref{intro.main.thm} for certain schemes in characteristic zero,
comprising varieties in characteristic zero, and for varieties in positive characteristic
under a strong assumption on resolution of singularities~\cite{Mor16}.
His approach is to
decompose the problem into cdh-descent for infinitesimal $K$-theory, which is due to
Haesemeyer \cite{Haes}, and into the analogous descent property for (topological) cyclic homology,
which is based on the homotopy cartesian square~\eqref{into.eq1}.

\smallskip

Combining Theorem~\ref{intro.main.thm} with the techniques of \cite{KerzStrunk16} we
obtain a proof of Weibel's conjecture on the vanishing of negative $K$-groups, formulated
as a question in~\cite[Qu.~2.9]{Wei80}.

\begin{introthm}\label{intro.thm.wc}
For a noetherian scheme $X$ of dimension $d<\infty$ the following hold.
\begin{itemize}
\item[(i)]\label{intro.thm.wc.i}
For $i<-d$ we have $K_i(X)=0$.
\item[(ii)]\label{intro.thm.wc.ii}
For $i\le -d$ and any integer $r\ge 0$ the map
\[
K_i(X) \xrightarrow{\cong} K_i(\mathbb A^r_X) 
\] 
is an isomorphism.
\end{itemize}
\end{introthm}

Weibel's conjecture, i.e.\ Theorem~\ref{intro.thm.wc}, has been known under one of the
following additional assumptions:
\begin{itemize}
\item[(i)] $d\le 1$, as can be deduced from \cite[Thm.~XII.10.4]{Bass68} (see the explanations in the paragraph following \eqref{intro.eq11} above),
\item[(ii)] $d=2$ and $X$ excellent \cite[Thm.~4.4]{Wei01},
\item[(iii)] $X/k$ a variety over a field $k$ with ${\rm ch}(k)=0$ \cite{CHSW},
\item[(iv)] $X/k$  a variety such that a strong form of resolution of singularities holds over
  the field $k$ \cite{GeiHe}, \cite{Krish} (only part of
Theorem~\ref{intro.thm.wc}). 
\end{itemize}

As further evidence, Kelly showed in~\cite{Kelly} that $K_i(X)\otimes \mathbb{Z}[1/p]=0$
for  $i<-d$ and  a
quasi-excellent scheme $X$ of dimension $d$  on which the prime $p$ is nilpotent; see also~\cite{KerzStrunk16} for
another proof of this result, which is very similar to the proof of
Theorem~\ref{intro.thm.wc} given in Subsection~\ref{subsec:weibel}.

Combining Theorem~\ref{intro.main.thm} and Theorem~\ref{intro.thm.wc} with a simple spectral
sequence argument we also obtain a new proof of cdh-descent for homotopy
$K$-theory, originally due to
Cisinski~\cite{Cisi13}. His proof relies on proper base change in stable motivic homotopy
theory. 

\begin{introthm}\label{intro.thm.kh}
For any abstract blow-up square~\eqref{intro.eq1} with $X$ finite dimensional the square of
homotopy $K$-theory spectra
\[
\xymatrix{
KH(X) \ar[r] \ar[d] &  KH(Y) \ar[d] \\
KH(\tilde X) \ar[r] & KH(E)
}
\]
is homotopy cartesian.
\end{introthm}

As a Corollary to Theorem~\ref{intro.thm.wc} and Theorem~\ref{intro.thm.kh} we get a
description of the lowest potentially non-vanishing $K$-group in terms of {\em
  cdh-cohomology}, see \cite[Def.~5.7]{SV}.

\begin{introcor}\label{intro.cor}
For a noetherian scheme $X$ of dimension $d<\infty$ there is a canonical isomorphism
\[
H^d_{\cdh}(X,\Z) \cong K_{-d}(X).
\]
\end{introcor}

For $X$ a variety in characteristic zero Corollary~\ref{intro.cor} is shown
in~\cite{CHSW}. In the course of proving it in Subsection~\ref{subsec.cdh} we also generalize the main result
of~\cite{Haes}, i.e.~the spectral sequence
\[
E^{p q}_2 = H^p_\cdh (X, \ra_\cdh K_{-q} ) \Rightarrow KH_{-p-q}(X),
\]
to any noetherian scheme $X$ of finite dimension, see Theorem~\ref{subsec.cdh.thm}. Here $\ra_\cdh$ is the
cdh-sheafification.

\subsection*{Overview}

In the proof of Theorem~\ref{intro.main.thm}
we essentially use two new methods. The first is a generalization of Thomason's descent
for algebraic $K$-theory of blow-ups in regularly immersed centers to the case of 
derived blow-ups, see Theorem~\ref{thm:descent-for-derived-blow-up}. The second is a
version of pro-excision for the $K$-theory of simplicial rings which allows us to pass from derived schemes back
to ordinary schemes, see Corollary~\ref{ex.corex2}.

Here is an overview of the proof of Theorem~\ref{intro.main.thm}.
Given a noetherian ring
$A'$, $X'=\Spec A'$, and a
regular sequence $\mathbf a'= (a'_0,\ldots , a'_r)\in (A')^{r+1}$ we consider the blow-up
$\tilde X'(2^n)$ 
with respect to the closed subscheme $Y'(2^n)$ corresponding to the ideal generated by
$(a'_0)^{2^n}, \ldots , (a'_r)^{2^n}$ and the blow-up square
\begin{equation}\label{intro.eq44}
\begin{split}
\xymatrix{
\tilde X'(2^n) \ar[d] &  \ar[l] E'(2^n) \ar[d]\phantom{.} \\ 
X' & \ar[l] Y'(2^n).
}
\end{split}
\end{equation}
For varying $n$ we get canonical finite transition morphisms from the $2^{n}$-square
to the $2^{n+1}$-square, so we get a whole tower of blow-up squares. We now consider an affine
noetherian scheme $X$ and take both the derived and the ordinary  pullback of the square~\eqref{intro.eq44} along a
morphism $X\to X'$. This derived pullback is the front square and the ordinary pullback is
the back square of  the commutative diagram
\begin{equation}\label{intro.eq45}
\begin{split}
 \xymatrix@R=2.5ex@C=1.2ex{ 
 \tilde X(2^n) \ar[dd] \ar[dr] &&  E(2^n) \ar[dr] \ar'[d][dd] \ar[ll] &  \\
& \tilde \dX(2^n)  \ar[dd]  &&   \ar[ll] \ar[dd]  \dE(2^n)\phantom{.}  \\
X  \ar@{=}[dr]   & & Y(2^n) \ar[dr]^\alpha \ar'[l][ll] &  \\
&  X &&  \ar[ll] \dY(2^n).  
}
\end{split}
\end{equation}
In this overview we would like to explain the proof of Theorem~\ref{intro.main.thm} for the
blow-up square
\begin{equation}\label{intro.eq46}
\begin{split}
\xymatrix{
 \Bl_{Y(1)} X\ar[d] & \ar[l] \ar[d] Y(1)\times_X \Bl_{Y(1)}X\\
X & \ar[l]  Y(1).
}
\end{split}
\end{equation}
Using pro-excision, see Proposition~\ref{pr.findesc}, we relate the
square~\eqref{intro.eq46} to the back square of~\eqref{intro.eq45}.
So we have to prove that the back square in~\eqref{intro.eq45} induces a cartesian square of
$K$-theory pro-spectra in the index $n$. The latter is established in the following
steps.
\begin{itemize}
\item[\it Step 1.]  For fixed $n$ the front square of~\eqref{intro.eq45} induces a cartesian square of
  $K$-theory spectra, see Theorem~\ref{thm:descent-for-derived-blow-up}.
\item[\it Step 2.] The map $\alpha$ induces an equivalence on $K$-theory pro-spectra in
  the index $n$, see Lemma~\ref{ex.lem.coclde2}.
\item[\it Step 3.] The upper square  induces a
cartesian square of $K$-theory pro-spectra in the index $n$. This is based on a version of pro-excision
for simplicial rings, see Corollary~\ref{ex.corex2} and Lemma~\ref{ex.lem.coclde2}.
\end{itemize}

\smallskip

In Section~\ref{sec:DerivedSchemes} we recall some foundations of the theory of derived
schemes and their perfect modules. Then we recall the Waldhausen--Thomason--Trobaugh
$K$-theory of derived schemes and in the affine case the relation to Waldhausen's plus
construction for simplicial rings. 
In Section~\ref{sec:DerivedBlowUp} we introduce our major
derived schemes, namely the Koszul scheme and the derived blow-up. We generalize
Thomason's descent for $K$-theory of blow-ups in regularly immersed centers to derived blow-ups in general sequences.
In Section~\ref{sec:ProExcision} we generalize part of Suslin's work on excision for
$K$-theory  to commutative simplicial rings.
Section~\ref{sec.proof} contains the reduction steps which lead to a proof of
Theorem~\ref{intro.main.thm}.
In Section~\ref{sec:application} we prove Theorem~\ref{intro.thm.wc},
Theorem~\ref{intro.thm.kh} and Corollary~\ref{intro.cor}.

\subsection*{Acknowledgement}

Ofer Gabber suggested to us to use derived algebraic geometry in order to generalize Thomason's
calculation of the $K$-theory of a blow-up in a regularly immersed center, see Section~\ref{sec:DerivedBlowUp}.
We would like to thank him for this essential remark.  We would like to thank Matthew
Morrow for several helpful discussions about his work on pro-calculations in $K$-theory.
Finally, we thank the referee for his or her careful reading of the manuscript and several comments to improve the presentation.

\subsection*{Notation}

If not explicitly stated otherwise a ring is commutative and unital. The object of a
pro-category corresponding to a system $(X_n)_{n\in I}$ indexed by a cofiltered category $I$
is denoted in Deligne's notation $\prolim{n} X_n$. All our indexing categories
of pro-objects are  countable. 
Weak equivalences of spaces or spectra are simply called equivalences.

A simplicial ring $A$ is called noetherian if $\pi_0 A$ is
a noetherian ring and $\pi_i A$ is a finitely generated $\pi_0 A$-module for all $i\ge 0$.
Schemes are denoted by italic letters $X,Y,\dots$,  derived schemes by calligraphic letters $\dX, \dY,  \dots$

By an $\infty$-category we mean more precisely an $(\infty,1)$-category. For the purpose of this paper, the choice of model of $\infty$-categories does not play any essential role. 
However, to be definite, and in accordance with main reference \cite{Blumberg-Tabuada-Gepner} for the $K$-theory of stable $\infty$-categories, we use the language of quasi-categories as developed by Joyal and Lurie \cite{htt}. 
We consider an ordinary category as an $\infty$-category via its nerve and abstain from
an extra decoration.

\section{Derived schemes and \texorpdfstring{$K$}{K}-theory}
\label{sec:DerivedSchemes}

We use this section to fix some notation for derived schemes and to recollect some facts about their $K$-theory that we will need.

\subsection{Simplicial rings and derived schemes}\label{subsec:simplicial-rings}
Basic references for the material of this subsection are \cite{ToenVezzosiHAGII} or  \cite{sag}.
An overview is \cite{ToenSurvey}.

Let $\cRing$ be the category of simplicial commutative rings.
A map in $\cRing$ is called a weak equivalence if the underlying map of simplicial sets is a weak equivalence.
We denote by $\sRing$ the $\infty$-category obtained from $\cRing$ by inverting the weak equivalences and call it the $\infty$-category of \emph{derived rings}.

\begin{defn}
  A \emph{derived scheme} is a pair $\dX = (|\dX|, \sO)$ consisting of a topological space
  $|\dX|$ and a sheaf of derived rings $\sO$ on $|\dX|$ which is
  hypercomplete\footnote{Note that every sheaf is hypercomplete if the topological space $|\dX|$ is noetherian of finite
    dimension.} as a sheaf of spaces satisfying the following conditions:
\begin{itemize}
\item[(i)] $(|\dX|, \pi_{0}\sO)$ is a scheme in the usual sense,
\item[(ii)] each of the sheaves $\pi_{n}\sO$ is quasi-coherent as a sheaf of $\pi_{0}\sO$-modules.
\end{itemize}
Here $\pi_{n}\sO$ denotes the sheaf of abelian groups associated to the presheaf $U \mapsto \pi_{n}(\sO(U))$ on $|\dX|$.
In fact, $\pi_{0}\sO$ is a sheaf of rings and each $\pi_{n}\sO$ is a sheaf of $\pi_{0}\sO$-modules.
\end{defn}
\begin{rmk}
Replacing `sheaf of derived rings' by `sheaf of $E_{\infty}$-ring spectra with $\pi_{n}\sO \simeq 0$ for $n<0$' in this definition, we obtain the notion of a \emph{spectral scheme} \cite[Def.~1.1.2.8]{sag}. Via the Eilenberg--MacLane functor, every derived scheme $\dX$ has an associated spectral scheme $H(\dX)$. 
\end{rmk}

Derived schemes are the objects of an $\infty$-category $\dSch$ (cf.~\cite[\S 1.1.5]{sag}). Since we do not need the precise construction of this $\infty$-category, we only describe its ($1$-)morphisms:
A morphism of derived schemes $\dX \to \dY$  is a pair  consisting of a morphism of topological spaces $f\colon |\dX| \to |\dY|$ and a morphism of sheaves $f^{\sharp}\colon \sO_{\dY} \to f_{*}\sO_{\dX}$ such that the induced morphism of ringed spaces $(|\dX|, \pi_{0}\sO_{\dX}) \to (|\dY|, \pi_{0}\sO_{\dY})$ 
 is a morphism of schemes, i.e.~is local.

The $\infty$-category $\dSch$ admits all finite limits.
If $\dX$ is a derived scheme, we denote the \emph{underlying scheme} $(|\dX|, \pi_{0}\sO)$ by $t\dX$.
Every ordinary scheme $X$ defines a derived scheme $iX$ with discrete structure sheaf. We will often omit $i$ from the notation. These constructions give rise 
to an adjunction
\begin{equation}\label{eq:constantderivedschemeadjunction}
i :  \Sch \leftrightarrows \dSch : t
\end{equation}
of $\infty$-categories, where $\Sch$ denotes the category of ordinary schemes.
As a right adjoint, $t$ sends  cartesian squares of derived schemes to cartesian squares of schemes. To avoid confusion, we will call cartesian squares in $\dSch$ `derived cartesian'. Given morphisms of derived schemes $\dX \to \dY$, $\dZ \to \dY$, we denote their fibre product by $\dX \times_{\dY}^{\dho} \dZ$. If $X,Y,Z$ are ordinary schemes, their derived fibre product $\dW =X\times_{Y}^{\dho} Z$ is a derived scheme with underlying scheme $X\times_{Y}Z$ and $\pi_{n}\sO_{\dW} \cong \mathcal{T}or_{n}^{\sO_{Y}}(\sO_{X},\sO_{Z})$ given by the higher Tor-sheaves.

A derived scheme $\dX$ is called \emph{affine} if $t\dX$ is an affine scheme.
Let $\dAff \subseteq \dSch$ denote the full $\infty$-subcategory spanned by the affine derived schemes.
The global sections functor induces an equivalence $\dAff \xrightarrow{\simeq} \sRing^{op}$ whose inverse is denoted by $\Spec$.
The latter functor can be described explicitly as follows:
Let $A$ be a simplicial ring.
The underlying space of $\Spec(A)$ is equal to that of the ordinary scheme $\Spec(\pi_{0}A)$, a closed subscheme of $\Spec(A_{0})$.
Now $A$ defines a sheaf of simplicial $A_{0}$-algebras on $|\Spec(A_0)|$ whose restriction to $|\Spec(\pi_{0}A)|$ is equivalent to the structure sheaf of the derived scheme $\Spec(A)$.
The  adjunction \eqref{eq:constantderivedschemeadjunction} restricts to the  adjunction $const:\Ring^{op}\leftrightarrows \sRing^{op}:\pi_0$ of $\infty$-categories.

\begin{defn}
A morphism of derived schemes $\dX \to \dY$ is called \emph{affine} if  for every map $\dZ \to \dY$ from an affine derived scheme the base change $\dX \times_{\dY}^{\dho} \dZ$ is an affine derived scheme, or, equivalently, if and only if the morphism of underlying schemes $t\dX \to t\dY$ is affine.
A derived scheme $\dX$ is called \emph{quasi-compact} if the topological space $|X|$ is quasi-compact.
It is called \emph{quasi-separated} if the diagonal morphism is quasi-compact, i.e.~pullbacks along morphisms from an affine derived scheme are quasi-compact.
We will use the abbreviation \emph{qcqs} for quasi-compact, quasi-separated.
\end{defn}

\begin{defn}
We call a  simplicial ring $A$ \emph{noetherian} if $\pi_{0}(A)$ is a noetherian ring and each $\pi_{n}(A)$ is a finitely generated $\pi_{0}(A)$-module. Similarly, we call a derived scheme $\dX$ \emph{noetherian} if the underlying scheme $t\dX$ is noetherian and each $\pi_{n}\sO$ is coherent as sheaf of $\pi_{0}\sO$-modules.
\end{defn}

Via the Eilenberg--MacLane functor, a simplicial commutative ring $A$ gives rise to an $E_{\infty}$-ring spectrum $HA$, and we denote by $\dMod(A)$ the stable presentable symmetric monoidal $\infty$-category of $HA$-module spectra \cite[\S 7.1]{halg}.

\begin{rmk}
A similar argument as in the proof of \cite[Thm.~7.1.2.13]{halg} yields
the following  more concrete description of $\Mod(A)$. The normalization $NA$ of $A$ is a commutative dg-algebra (cf.~\cite{SchwedeShipley03}). The $\infty$-category $\Mod(A)$ is equivalent to the $\infty$-category of unbounded $NA$-dg-modules with  quasi-isomorphisms inverted.
\end{rmk}

For a derived scheme $\dX$ we denote by $\QCoh(\dX)$ the stable presentable symmetric monoidal $\infty$-category of quasi-coherent $\sO_{\dX}$-modules \cite[Def.~2.2.2.1]{sag}. In fact, $\QCoh(\dX)$ is equivalent to the $\infty$-category of quasi-coherent sheaves on the spectral scheme $H(\dX)$ (cf.~\cite[Cor.~2.15]{Shipley02}).  If $\dX=\Spec(A)$ is affine, we have an equivalence $\QCoh(\dX)\simeq \dMod(A)$. In general, it can be described as the $\infty$-categorical limit
\[
\QCoh(\dX) \simeq \lim_{\Spec(A) \in (\dAff/\dX)^{op}} \dMod(A)
\]
of $\infty$-categories.
This generalizes the classical notion: If $X$ is an ordinary scheme, the homotopy category $\Ho(\QCoh(iX))$ is equivalent to the unbounded derived category of  $\sO_{X}$-modules with quasi-coherent cohomology sheaves \cite[Cor.~2.2.6.2]{sag}. 
As a stable $\infty$-category $\QCoh(\dX)$ admits mapping spectra which we denote by $\Hom(-,-)$.

The following lemma is well known, see \cite[Prop.~3.6]{BFN} or \cite[Prop.~6.2.6.2]{sag}.
\begin{lem}
An object $F \in \QCoh(\dX)$ is dualizable if and only if it is \emph{perfect}, i.e.~its restriction to each affine $\Spec(A)\to \dX$ lies in the the smallest idempotent complete stable $\infty$-subcategory of $\Mod(A)$ containing $A$ as an object.
\end{lem}

\begin{defn}
We denote by $\Perf(\dX)$ the full stable idempotent complete $\infty$-subcategory  of $\QCoh(\dX)$ consisting of the perfect modules.
\end{defn}

\begin{rmk}\label{rmk:perfect-complexes}
If $X$ is an ordinary scheme, then the homotopy category $\Ho(\Perf(iX))$ is equivalent to the usual derived category of perfect complexes of $\sO_{X}$-modules.
\end{rmk}

\begin{prop}
Let $\dX$ be a qcqs derived scheme. Then $\QCoh(\dX)$ is compactly generated, and the compact objects coincide with the perfect ones.
\end{prop}
This is \cite[Prop.~9.6.1.1]{sag}. Note that, since the $\infty$-categories in question are stable, one can check whether they are generated by  a set  $S$ of objects on the level of (triangulated) homotopy categories  \cite[Rem.~1.4.4.3]{halg}.

\subsection{\texorpdfstring{$K$}{K}-theory of derived schemes}
Using $\infty$-categorical analogs of Waldhausen's $S_{\bullet}$-construction and Schlichting's construction of non-connective $K$-theory of exact and derived categories, Blumberg--Gepner--Tabuada introduce non-connective $K$-theory of small stable $\infty$-categories in \cite[\S 9.1]{Blumberg-Tabuada-Gepner}. If $\sC$ is such an $\infty$-category, we denote its non-connective $K$-theory spectrum by $K(\sC)$. 

Let $\Cat^{\mathrm{ex}}_{\infty}$ be the $\infty$-category of small stable $\infty$-categories and exact functors.
If $\sA \to \sB$ is a fully faithful exact functor between small stable \mbox{$\infty$-categories}, one can form its cofibre $\sB/\sA$ in $\Cat^{\mathrm{ex}}_{\infty}$. This is a model for the Verdier quotient: $\Ho(\sB)/\Ho(\sA) \xrightarrow{\simeq} \Ho(\sB/\sA)$ (see~\cite[Prop.~5.14]{Blumberg-Tabuada-Gepner}).
A sequence $\sA \to \sB \to \sC$ in $\Cat^{\mathrm{ex}}_{\infty}$ is called \emph{exact} if the composite is trivial, $\sA \to \sB$ is fully faithful, and the induced functor $\sB/\sA \to \sC$ becomes an equivalence after idempotent completion; this is equivalent to requiring that $\Ho(\sA) \to \Ho(\sB) \to \Ho(\sC)$ be an exact sequence of triangulated categories \cite[Prop.~5.15]{Blumberg-Tabuada-Gepner}.

A direct consequence of \cite[Thm.~9.8]{Blumberg-Tabuada-Gepner} is:
\begin{thm} \label{thm:K-is-localizing}
If $\sA \to \sB \to \sC$ is an exact sequence in $\Cat^{\mathrm{ex}}_{\infty}$, then the sequence of non-connective $K$-theory spectra
\[
K(\sA) \to K(\sB) \to K(\sC)
\]
is a cofibre sequence.
\end{thm}

\begin{defn}
Let $\dX$ be a qcqs derived scheme. We define its $K$-theory spectrum by $K(\dX) = K(\Perf(\dX))$. If $\dX = \Spec(A)$ is affine, we write $K(A)= K(\dX)$ for simplicity.
\end{defn}
Using Remark~\ref{rmk:perfect-complexes} we see that, for $X$ an ordinary qcqs scheme, $K(iX)$ is equivalent to the $K$-theory spectrum constructed in \cite{Thomason-Trobaugh}.

In \cite[Prop.~A.13]{ClausenMathewNaumannNoel16}, the authors generalize a descent result of Thomason to the setting of spectral algebraic spaces. In our setting it gives:

\begin{thm}\label{thm:Zariski-descent}
The $K$-theory of derived schemes with underlying noetherian topological space satisfies \v{C}ech-descent for the Zariski topology. 
\end{thm}

To prove the excision result in Section~\ref{sec:ProExcision} we need the description of the connective part of the $K$-theory of affine derived schemes via the plus-construction.
Following Waldhausen \cite{Waldhausen87}, we consider for an integer $n\geq 0$ and for a
not necessarily unital simplicial ring $I$ the group-like simplicial monoid $\GLh(I)_n$
defined by the cartesian square
\begin{equation}\label{eq:waldhausenglhatconstruction}
\begin{split}
\xymatrix{
\GLh_n(I) \ar[d]\ar[r] & \rM_{n,n}(I) \ar[d] \\
\GL_n(\pi_0 I)\ar[r] &  \rM_{n,n}(\pi_0 I)
}
\end{split}
\end{equation}
which is also homotopy cartesian, as the right vertical map is a fibration. Here, as usual
for a nonunital ring $J$, the group $\GL_n(J)$ is defined by
\[
\GL_n(J) = \ker (  \GL_n(\tilde J )\to \GL_n(\Z ) )
\]
where $\tilde J =  \Z \ltimes J$ is the unitalization of $J$.

Taking the colimit over $n$ of the diagram \eqref{eq:waldhausenglhatconstruction}, one
obtains the cartesian and homotopy cartesian square
\begin{equation}\label{eq:waldhausenglhatconstructionforideals}
\begin{split}
\xymatrix{
\GLh(I) \ar[d]\ar[r] & \rM (I) \ar[d]\phantom{.} \\
\GL(\pi_0 I)\ar[r] &  \rM (\pi_0 I).
} 
\end{split}
\end{equation}
Clearly, $\pi_0( \GLh(I)) \cong \GL(\pi_0 I) $. As the upper horizontal map
in~\eqref{eq:waldhausenglhatconstructionforideals} induces an isomorphism of connected
components of the identity elements, we get $\pi_i( \GLh(I) ) \cong \rM (\pi_i I) $ for $i>0$.
In particular, the simplicial monoid $\GLh(I)$ is grouplike. From the classical \cite[Prop.~1.5]{Segal74} one obtains

\begin{lem}\label{lemma:segalstheoremappliedtobglhat}
The connected classifying space $\BGLh(I)$ has homotopy groups 
\[
\pi_i(\BGLh(I))\cong
  \pi_{i-1}(\GLh (I)) \cong   \begin{cases} 
   \GL(\pi_0(I))  & \text{if } i=1\\
   \rM(\pi_{i-1}(I)) & \text{if } i>1.
  \end{cases} 
\] 
\end{lem}

For a unital simplicial ring $A$ 
we can apply Quillen's plus construction \cite[IV.1]{Weib13} and obtain a natural transformation
$\BGLh(A)\to \BGLh(A)^+$ which is characterized up to homotopy by the fact that it is
acyclic and that it kills the commutator subgroup of $ \pi_1 ( \BGLh(A)) \cong \GL(\pi_0
A) $. So we get

\begin{lem}\label{ds.lemk1}
 The canonical map
\[
\pi_1 (\BGLh(A)^+ )\xrightarrow{\cong}  \pi_1 ( \BGL(\pi_0 A)^+ )\cong K_1( \pi_0 A)
\]
is an isomorphism.
\end{lem}

The following result follows directly from Waldhausen \cite[Thm.~2.3.2]{Waldhausen85} in conjunction with \cite[Prop.~9.31, 9.32]{Blumberg-Tabuada-Gepner}.
Alternatively, see \cite[Lemma~9.39, Prop.~9.40]{Blumberg-Tabuada-Gepner}.
\begin{prop}
\label{prop:plusequalswaldhausen}
The infinite loop space associated to the spectrum $K(A)$ is equivalent to $K_0(\pi_0(A))\times B\GLh(A)^+$. 
\end{prop}

As an application, we get:
\begin{thm}\label{thm:identificationofthenegativekgroups}
For an affine derived scheme $\Spec(A)$, the natural map \[\Spec(\pi_{0}A) \cong t\Spec(A) \to \Spec(A)\] induces isomorphisms $K_{n}(A) \cong K_{n}(\pi_0 A)$ for all integers $n\leq 1$.
\end{thm}
In \cite[Thm.~9.53]{Blumberg-Tabuada-Gepner} this is proved for connective $A_{\infty}$-ring spectra.
We offer the following alternative proof, which works for simplicial commutative rings, using some results of Section~\ref{sec:DerivedBlowUp}, because it fits nicely in the framework of the current paper.
\begin{proof}
By Proposition~\ref{prop:plusequalswaldhausen}, the claim is true for $n=0$. By
Lemma~\ref{ds.lemk1} it is true for $n=1$.
Since both versions of $K$-theory satisfy the Bass fundamental theorem (see Theorem~\ref{thm:Bass-fundamental} for the version for derived schemes), we deduce that the map in question is an isomorphism for all $n\leq 1$ by induction.
\end{proof}

\begin{lem}\label{lemma:nilinvarianceinsmalldegrees}
For a noetherian derived scheme $\dX$ there exists $i_0 \in \mathbb Z$ such that $K_i(\dX)\to K_i (t(\dX)_{\rm red}) $ is an isomorphism for $i\le i_0$.
\end{lem}
\begin{proof}
We reduce the statement of the lemma first to separated and then to affine derived schemes. By the noetherian assumption, we find a finite Zariski covering $\mathcal{U}$ of $\dX$ by affine derived schemes.
In particular, every possible intersection of the involved open subschemes is separated. As $K$-theory satisfies Zariski descent (Theorem~\ref{thm:Zariski-descent}),
 there is a convergent spectral sequence
\[
 E_2^{pq}=H^p(\mathcal{U},K_{-q}) \Rightarrow K_{-p-q}(\dX).
\]
Hence the statement of the lemma is reduced to separated derived schemes. Now each of the finitely many intersections has a finite Zariski covering such that all of its intersections are affine.
This reduces the statement to affine derived schemes  by the same argument as above.

In the affine situation, the statement (with $i_{0}=0$) is a direct consequence of the previous Theorem \ref{thm:identificationofthenegativekgroups} and nil-invariance for non-positive algebraic $K$-theory of discrete rings.
The latter follows from the definition of negative $K$-theory \cite[Def.~III.4.1]{Weib13} through $K_0$ and nil-invariance of $K_0$ for rings \cite[Lemma~II.2.2]{Weib13}. 
\end{proof}

\section{Derived blow-ups}
\label{sec:DerivedBlowUp}

In this section we define  derived blow-ups of  affine schemes and prove a descent theorem for derived blow-up squares similar to Thomason's descent theorem for blow-ups in regularly immersed centers.
Using the same techniques we also prove a projective bundle theorem and a Bass fundamental theorem for derived schemes.

\subsection{Derived blow-up squares}\label{subsec:derived-blow-up-squares}

Consider a noetherian ring $A$ and  a finite sequence $\ba=(a_{0},\dots, a_{r}) \in A^{r+1}$.
Write $X=\Spec(A)$. In the following, we introduce the notion of the derived blow-up $\Bld_{\ba}(X)$ of $X$  and the Koszul derived subscheme $\dV(\ba)$ of $X$ associated to the sequence $\ba$.

For this we choose an auxiliary noetherian ring $A'$ and a regular sequence $\ba' \in A^{\prime, r+1}$ together with a map $f\colon A' \to A$ sending $\ba'$ to $\ba$ (e.g.~a polynomial ring).
Let $Y' = V((\ba')) \subseteq X'=\Spec(A')$, $\tilde X'$ the blow-up of $X'$ along $Y'$, $E'$ the exceptional divisor:
\begin{equation}\label{diag:regular-blow-up}
\begin{split}
\xymatrix{
\tilde X'  \ar[d]_{p'} &  E'  \ar[l]_{j'} \ar[d]^{q'} \\
X'  & Y' \ar[l]_{i'}
}
\end{split}
\end{equation}

\begin{defn}\label{dfn:derived-pullback}
We denote the derived pullback of \eqref{diag:regular-blow-up} along $f\colon X \to X'$ by
\begin{equation}\label{diag:derived-blow-up}
\begin{split}
\xymatrix{
 \tilde \dX \ar[d]_{p} & \dE \ar[l]_{j} \ar[d]^{q}  \\
X & \dY \ar[l]_{i} 
}
\end{split}
\end{equation}
and call $\Bld_{\ba}(X) = \tilde\dX$ the \emph{derived blow-up} of $X$ in $\ba$ and $\dV(\ba) = \dY$ the \emph{Koszul derived scheme} associated to $\ba$.
A square of the form \eqref{diag:derived-blow-up} is called \emph{derived blow-up square}.
We call $\dE$ the \emph{semi-derived exceptional divisor} and write $\dD = \tilde\dX \times_{X}^h \dY$ for  the \emph{derived exceptional divisor}.

We write $\Ko(A; \ba) $ for the derived ring corresponding to the affine derived scheme $\dV(\ba)$.
Its homotopy groups are the usual  Koszul homology groups of the sequence $\ba$.
\end{defn}

\begin{rmk}\label{rmk:helpfulremark}
The following diagram might help the reader to relate the terminology introduced in the previous Definition~\ref{dfn:derived-pullback}.
\begin{equation}\label{diag:helpfulsquare}
\begin{split}
 \xymatrix@R=2.5ex@C=1.2ex{ 
              && \tilde\dX \ar[ddd] \ar[drr]  &  &                           &   &&  \dE \ar[drr] \ar'[d][ddd] \ar[lllll]    &&                                         \\
              &&                             &   & \tilde X'  \ar[ddd]      &   &&                                      &&  \ar[lllll] \ar[ddd] E'     \\
X\phantom{'}\ar@{=}[drr]  &&                             &   &            \ar'[ll][llll]               & Y\phantom{'} \ar@{-}[l]\ar[drr] &&                                      &&                                         \\
              && X\phantom{'}  \ar@{->}[drr]_(0.4)f             &   &                           &   && \dY\phantom{'} \ar[drr] \ar'[lll][lllll]    &&                                         \\
              &&                             &   &  X'                       &   &&                                      &&  \ar[lllll] Y'            
}
\end{split}
\end{equation} 
Here we set $Y=V((\ba))$.
The left, the right, the bottom and the top faces of this cube are derived cartesian, whereas the front and the back face need not to be derived cartesian.
In particular, the morphism $\dE\to\dD$ is not an equivalence in general, even though the underlying schemes $t\dE\xrightarrow{\simeq}t\dD$ are isomorphic.
\end{rmk}

We will also consider the following thickenings. For $n\in \N$ write $\ba(n)=(a_{0}^{n}, \dots, a_{r}^{n})\in A^{r+1}$. 
In the situation above we write $Y'(n) = V((\ba'(n)))$ 
and $E'(n)= \tilde X' \times_{X'} Y'(n)$ for the corresponding thickening of the exceptional divisor.
As before we define $\dY(n)$ and $\dE(n)$ by a derived pullback along $X\to X'$. Note that $\dY(n)$ is just the Koszul derived scheme associated to the sequence $\ba(n)$.
Finally, we write $\dD(n) = \tilde\dX \times_{X}^{h} \dY(n)$ for the thickened derived exceptional divisor.

\begin{rmk}\label{rmk:ordinary-and-derived-blow-up}
We always have a closed immersion of the ordinary blow-up $\Bl_{(\ba)}(X)$ into the underlying scheme  $t(\Bld_{\ba}(X)) \cong X \times_{X'} \Bl_{(\ba')}(X')$ of the derived blow-up.
\end{rmk}

In the remainder of this subsection we prove that, up to equivalence, the notions introduced above do not depend on the choices made.

\begin{lem}\label{lem:V-derived-cartesian}
Let $A'\to A$ be a homomorphism of rings sending the regular sequence $\ba'$ to the regular sequence $\ba$. Then the diagram of schemes
\[
\xymatrix{
V((\ba)) \ar[r] \ar[d] & V((\ba')) \ar[d] \\
\Spec(A) \ar[r] & \Spec(A')
}
\]
is derived cartesian.
\end{lem}
\begin{proof}
The square is cartesian as a square of schemes. Hence it suffices to show that $\Tor^{A'}_{*}(A, A'/(\ba'))$ vanishes in positive degrees.
Since $\ba'$ is a regular sequence, we can compute these using the Koszul complex of $\ba'$. Tensoring with $A$ then gives the Koszul complex of the sequence $\ba$. It is exact in positive degrees, since $\ba$ was assumed to be regular, too.
\end{proof}

\begin{lem}
In the situation of Lemma~\ref{lem:V-derived-cartesian}, write $\tilde X = \Bl_{V((\ba))}X$ for the ordinary blow-up of
$X=\Spec(A)$ in $V((\ba))$,  $E$ for the exceptional divisor, and $E(n)$ for its thickenings defined by the sequence $\ba(n)$,  and similarly for the primed version. Then the diagrams 
\[
\xymatrix{
\tilde X \ar[r]\ar[d]_{p} & \tilde X' \ar[d]^{p'} \\
X \ar[r] & X'
} \quad \text{ and } \quad
\xymatrix{
E(n) \ar[r]\ar[d] & E'(n) \ar[d] \\
X \ar[r] & X'
}
\]
are derived cartesian.
\end{lem}
\begin{proof}
We begin with the left-hand square.
Since $\ba$ is a regular sequence, we have 
$\tilde X = \Proj(A[T_{0}, \dots, T_{r}]/(a_{i}T_{j} - a_{j}T_{i})_{1\leq i,j\leq r})$ \cite[Ch.~I, Thm.~1]{Micali} and similarly for $\tilde X'$, and we see that $\tilde X \cong X \times_{X'} \tilde X'$.

To prove that the square is derived cartesian, we  argue locally. Consider a point $x\in X$ with image $x' \in X'$. If $x\not\in V((\ba))$, and hence $x' \not\in V((\ba'))$, then $p$ respectively $p'$ is an isomorphism near $x$ respectively $x'$, and we are done.
Now assume that $x\in V((\ba))$. Then the image of the sequence $\ba$ in the localization of $A$ at $x$ is still a regular sequence. The analog is true for $x', A'$, and $\ba'$. Hence we may assume that $A, A'$ are noetherian local rings. Then any permutation of the sequence $\ba$ is again a regular sequence. Hence we can deduce from \cite[Lem.~A.6.1]{Fulton} that  the standard affine open subset $U_{i} \cong\Spec(A[S_{0}, \dots,\hat S_{i} ,\dots, S_{r}]/(a_{i}S_{j} - a_{j})_{0\leq j \leq r, j\not=i} )$ of $\tilde X$ where  $\sO_{\tilde X}(1)$ is generated by $a_{i}$ is defined by the regular sequence $(a_{i}S_{0}-a_{0}, \dots)$ in $A[S_{0}, \dots, \hat S_{i}, \dots, S_{r}]$.
Again, the same holds with primes.
From Lemma~\ref{lem:V-derived-cartesian} we deduce that
\[
\xymatrix{
U_{i} \ar[d]\ar[r] & U'_{i} \ar[d] \\
\A^{r}_{X} \ar[r] & \A^{r}_{X'}
}
\]
is derived cartesian. This implies the statement for the first square, since $\A^{r}_{X'} \to X'$ is flat.

For the second square  we observe that $E$  and also its thickening $E(n)$ is locally cut out in $\tilde X$ by one regular element $a_{i}$ respectively $a_{i}^{n}$, and similar for $E', E'(n)$. We conclude by what we have already shown and Lemma~\ref{lem:V-derived-cartesian}.
\end{proof}

Summarizing we have:
\begin{lem}\label{lem:independence}
Up to equivalence, the Koszul derived scheme $\dY=\dV(\ba)$, the derived blow-up square \eqref{diag:derived-blow-up},  the thickenings $\dY(n), \dE(n)$, and hence $\dD(n)$ do not depend on  the choice of  $A'$ and $\ba'$.
\end{lem}
\begin{proof}
Given two rings $A'_{1}, A'_{2}$ with maps $A'_{i} \to A$ and regular sequences $\ba'_{i}$ mapping to $\ba$, we find a third ring $A'_{3}$ with maps $A'_{3} \to A'_{i}, i=1,2$, and a regular sequence $\ba'_{3}$ mapping to $\ba'_{1}, \ba'_{2}$ respectively (take $A'_{3} = \Z[T_{0}, \dots, T_{r}]$, $\ba'_{3}=(T_{0}, \dots, T_{r})$).
Hence the claim follows from the preceding lemmas.
\end{proof}

\subsection{Descent for derived blow-up squares}

The goal of this subsection is to prove the following descent statement for the derived
blow up square \eqref{diag:derived-blow-up}. It generalizes Thomason's descent theorem for
blow-ups in regularly immersed centers \cite{Thomason-blowup}.
\begin{thm}\label{thm:descent-for-derived-blow-up}
The square of non-connective $K$-theory spectra
\[
\xymatrix{
K(\tilde \dX) \ar[r]^{j^{*}} & K(\dE) \\
K(X) \ar[u]^{p^{*}} \ar[r]^{i^{*}} & K(\dY) \ar[u]_{q^{*}}
}
\]
is homotopy cartesian.
\end{thm}
The proof follows closely \cite[Sec.~1]{CHSW} with a few additional twists.
If $\sD$ is a small stable $\infty$-category and  $S$ is a set of objects of $\sD$, we write $\langle S \rangle \subseteq \sD$ for the smallest full stable subcategory of $\sD$ which contains $S$ and is closed under retracts in $\sD$. We say that $S$ generates $\sD$ if $\langle S \rangle = \sD$.

A useful criterion is the following.
Assume that  $\sC$ is a compactly generated stable $\infty$-category \cite[Def.~5.5.7.1]{htt}. Write $\sC^{\omega}$ for the small stable idempotent complete subcategory of its compact objects. Then $\Ind(\sC^{\omega}) \simeq \sC$. If $S$ is a set of objects of $\sC^{\omega}$, we have $\langle S \rangle \subseteq \sC^{\omega}$, and this  induces a fully faithful functor $i\colon \Ind(\langle S \rangle) \to \Ind(\sC^{\omega}) \simeq \sC$. The functor $i$ admits a right adjoint~$r$ \cite[Prop.~5.3.5.13]{htt}. 
If the right orthogonal of $S$ in $\sC$ vanishes, then $r$ is fully faithful, too, hence $i$ is an equivalence, hence $\langle S \rangle \simeq \sC^{\omega}$ \cite[Lem.~5.4.2.4]{htt}, i.e.~$S$ generates $\sC^{\omega}$.

\begin{lem}\label{lem:pullbacks-generate}
Let $f\colon \dX\to \dY$ be an affine morphism of qcqs derived schemes and let  $S$ be a set of perfect modules that generates $\Perf(\dY)$. Then the set $\{f^{*}F\,|\,F\in S\}$ generates $\Perf(\dX)$.
\end{lem}
\begin{proof}
First note that $f^{*}$ preserves perfect modules.

Now take any $G\in \QCoh(\dX)$ and assume that $\Hom(f^{*}F, G) \simeq 0$ for all $F\in S$. By adjunction we get $\Hom(F, f_{*}G) \simeq 0$ for all $F\in S$, 
and hence $f_{*}G \simeq 0$.

We claim that $G\simeq 0$. This is a local question, so we may assume that $\dY$ and hence $\dX$ are affine, say $\dX=\Spec(B), \dY=\Spec(A)$. In this case $\QCoh(\dX) \simeq \Mod(B)$ and similarly for $\dY$, and $f_{*}$ is equivalent to the forgetful functor $\Mod(B) \to \Mod(A)$, which detects 0-objects.  
\end{proof}

We now consider the derived blow-up square \eqref{diag:derived-blow-up} and retain the notation from the beginning of Subsection~\ref{subsec:derived-blow-up-squares}.
We denote the morphisms induced from $f\colon X \to X'$ by derived base change by $\tilde f\colon \tilde \dX \to \tilde X'$, $f_{\dY}\colon \dY \to Y'$, and \mbox{$f_{\dE}\colon \dE \to E'$}, respectively.
Since affine morphisms are stable under base change, all these are affine.

We write $\sO_{\tilde \dX}(1) = \tilde f^{*}\sO_{\tilde X'}(1)$, etc.
\begin{lem}\label{lem:explicit-generators} We have:
\begin{enumerate}
\item[(i)]  $\Perf(\tilde \dX)$ is generated by $\sO_{\tilde \dX}, j_{*}\sO_{\dE} \otimes \sO_{\tilde \dX}(-l)$ for $l=1, \dots, r$.
\item[(ii)] $\Perf(\dE)$  is generated by $\sO_{\dE}(-l)$, for $l=0, \dots, r$.
\end{enumerate}
\end{lem}
\begin{proof}
  By \cite[Lemma 1.2]{CHSW} the corresponding statements are true for the underived
  blow-up in diagram~\eqref{diag:regular-blow-up} (use that in our situation $X'$ and $Y'$
  are affine, hence $\Perf(X')$ is generated by $\sO_{X'}$ and $\Perf(Y')$ by $ \sO_{Y'}$). Using Lemma~\ref{lem:pullbacks-generate} we deduce that
  $\Perf(\tilde \dX)$ is generated by $\sO_{\tilde \dX}$ and
\[
 \tilde f^{*}(j'_{*}\sO_{E'} \otimes \sO_{\tilde X'}(-l)) \simeq j_{*}\sO_{\dE} \otimes \sO_{\tilde \dX}(-l)
\]
for  $l=1, \dots , r$. Here we used that $\tilde f^{*}$ is symmetric monoidal and base change \cite[Cor.~3.4.2.2]{sag} for the derived cartesian square
\[
\xymatrix{
\dE \ar[r]^{f_{\dE}} \ar[d]_-{j}     & E' \ar[d]^-{j'}\phantom{.} \\
\tilde \dX \ar[r]^{\tilde f} & \tilde X'.
}
\]
This proves (i). Even easier we get (ii).
\end{proof}

\begin{lem}\label{lem:fully-faithful}
The functors
$p^{*}\colon \Perf(X) \to \Perf(\tilde \dX)$ and
$q^{*}\colon \Perf(\dY) \to \Perf(\dE)$
are fully faithful.
\end{lem}
\begin{proof}
For the first functor, we may equivalently show that the unit $1 \to p_{*}p^{*}$ is an equivalence. Since $f\colon X \to X'$ is affine, it is enough to prove that $f_{*} \to f_{*}p_{*}p^{*}$ is an equivalence.
Using again base change for the derived cartesian square
\[
\xymatrix{
\tilde \dX \ar[r]^{\tilde f} \ar[d]_-{p} & \tilde X' \ar[d]^-{p'} \\
X \ar[r]^{f} & X'
}
\]
we have $f_{*}p_{*}p^{*} \simeq p'_{*}\tilde f_{*}p^{*} \simeq p'_{*}p^{\prime *}f_{*}$ and the claim follows because $1 \to p'_{*}p^{\prime *}$ is an equivalence \cite[Lemme 2.3]{Thomason-blowup}.

The same argument works for $q^{*}$.
\end{proof}

We have the fundamental fibre sequence
\(
\sO_{\tilde X'}(1) \to \sO_{\tilde X'} \to j'_{*}(\sO_{E'})
\)
in $\Perf(\tilde X')$. Pullback via $\tilde f$ and base change $\tilde f^{*}j'_{*}\simeq j_{*}f_{E}^{*}$  give the fibre sequence
\(
\sO_{\tilde \dX}(1) \to \sO_{\tilde \dX} \to j_{*}\sO_{\dE}
\)
in $\Perf(\tilde \dX)$. After we apply $j^{*}$ to this fibre sequence the last map has a retraction given by the counit $j^{*}j_{*}\sO_{\dE} \to \sO_{\dE}$. Hence we see that
\begin{equation}\label{eq:jupperstarjlowerstar}
 j^{*}j_{*}\sO_{\dE} \simeq \sO_{\dE} \oplus \sO_{\dE}(1)[1]
\end{equation}
in $\Perf(\dE)$.

\begin{lem}\label{lem:vanishing}
For $1 \leq k \leq r$ we have $\Hom(\sO_{\dE}, \sO_{\dE}(-k)) \simeq 0$.
\end{lem}
\begin{proof}
Let $B'$ respectively $B$ be the (derived) rings defining the affine (derived) schemes $Y'$ respectively $\dY$. 
Note that $B'$ is an ordinary ring, and $E' = \P^{r}_{B'}$ is the ordinary projective space.

The functor $\Hom(\sO_{\dE}, -)\colon \QCoh(\dE) \to \Sp$ is equivalent to the composition
\[
\QCoh(\dE) \xrightarrow{q_{*}} \QCoh(\dY) \xrightarrow{f_{\dY*}} \QCoh(Y')  \xrightarrow{\Hom(\sO_{Y'},-)} \Sp.
\]
We have 
\begin{align*}
f_{\dY*}q_{*}\sO_{\dE}(-k) &\simeq f_{\dY*}q_{*}f_{\dE}^{*}\sO_{E'}(-k) \\
&\simeq f_{\dY*}f_{\dY}^{*}q'_{*}\sO_{E'}(-k)
\end{align*}
using base change again. Finally, $q'_{*}\sO_{E'}(-k) \simeq 0$ by the classical computation of the cohomology of projective space \cite[Cor.~III.2.1.13]{EGA}.
\end{proof}

\begin{lem}\label{lemma:fullyfaithfulnessofjlowerstarpupperstar}
The functor $j_{*}q^{*}\colon \Perf(\dY) \to \Perf(\tilde \dX)$ is fully faithful.
\end{lem}
\begin{proof}
We have to show that $\Hom(F,G) \to \Hom(j_{*}q^{*}F, j_{*}q^{*}G)$ is an equivalence for all $F,G\in \Perf(\dY)$. Since $\Perf(\dY)$ is generated by $\sO_{\dY}$, it suffices to show this for $F=G=\sO_{\dY}$.
We have $q^{*}\sO_{\dY}\simeq \sO_{\dE}$ and
\begin{align*}
 \Hom(j_{*}q^{*}\sO_{\dY}, j_{*}q^{*}\sO_{\dY}) &\simeq  \Hom(j_{*}\sO_{\dE}, j_{*}\sO_{\dE}) \\
 &\simeq \Hom(j^{*}j_{*}\sO_{\dE}, \sO_{\dE}) \\
 & \simeq \Hom(\sO_{\dE} \oplus \sO_{\dE}(1)[1], \sO_{\dE}) \\
 & \simeq \Hom(\sO_{\dE}, \sO_{\dE}) \oplus \Hom(\sO_{\dE}, \sO_{\dE}(-1)[-1]) \\
 & \simeq \Hom(\sO_{\dY}, \sO_{\dY})
\end{align*}
by \eqref{eq:jupperstarjlowerstar}, the fully faithfulness of $q^{*}$, and since $\Hom(\sO_{\dE}, \sO_{\dE}(-1)[-1]) \simeq 0$ by Lemma \ref{lem:vanishing}.
\end{proof}

For $l=0,\dots, r$, let 
\[
\Perf^{l}(\tilde \dX) \subseteq \Perf(\tilde \dX)
\]
 be the full stable idempotent complete subcategory generated by $\sO_{\tilde \dX}$ and $j_{*}\sO_{\dE} \otimes \sO_{\tilde \dX}(-k)$ for $k=1, \dots, l$.
Let
\[
\Perf^{l}(\dE) \subseteq \Perf(\dE)
\]
be the full stable idempotent complete subcategory generated by $ \sO_{\dE}(-k)$ for $k=0, \dots, l$.

By Lemma~\ref{lem:explicit-generators} we have $\Perf^{r} = \Perf$ in both cases. By Lemma~\ref{lem:fully-faithful} we have
equivalences $p^{*}\colon \Perf(X) \xrightarrow{\simeq} \Perf^{0}(\tilde \dX)$ and $q^{*}\colon \Perf(\dY) \xrightarrow{\simeq} \Perf^{0}(\dE)$.
Moreover, $j^{*}\colon \Perf(\tilde \dX) \to \Perf(\dE)$ respects the filtrations by the $\Perf^{l}$. Indeed, it suffices to check this on generators. We have $j^{*}(j_{*}\sO_{\dE}\otimes \sO_{\tilde \dX}(-k)) \simeq j^{*}j_{*}\sO_{\dE} \otimes \sO_{\dE}(-k) \simeq \sO_{\dE}(-k) \oplus \sO_{\dE}(-k+1)[1]$ by \eqref{eq:jupperstarjlowerstar}.

\begin{lem}\label{lem:technical-lemma-1}
The composite functor 
\[
\Perf(\dY) \xrightarrow{q^{*}(-) \otimes \sO_{\dE}(-l)} \Perf^{l}(\dE) \to \Perf^{l}(\dE)/\Perf^{l-1}(\dE)
\]
is an equivalence.
\end{lem}
\begin{proof}
The composition sends the generator $\sO_{\dY}$ to the generator $\sO_{\dE}(-l)$. Since $\Perf(\dY)$ is idempotent complete, it is now enough to show that the composition is fully faithful.

The first functor is fully faithful by Lemma \ref{lem:fully-faithful}. By \cite[Lem.~9.1.5]{Neeman-TriCat} 
it now suffices to show that $\Hom(F, q^{*}G\otimes \sO_{\dE}(-l)) \simeq 0$ for every $G\in \Perf(\dY)$ and $F\in \Perf^{l-1}(\dE)$.
It is again enough to check this on generators. But for $k=0, \dots, l-1$ we have
\[
\Hom(\sO_{\dE}(-k), \sO_{\dE}(-l)) \simeq \Hom(\sO_{\dE}, \sO_{\dE}(k-l)) \simeq 0
\]
by Lemma \ref{lem:vanishing}.
\end{proof}

\begin{lem}\label{lem:technical-lemma-2}
The composition
\[
\Perf(\dY) \xrightarrow{j_{*}q^{*}(-)\otimes \sO_{\tilde \dX}(-l)} \Perf^{l}(\tilde \dX) \to \Perf^{l}(\tilde \dX)/\Perf^{l-1}(\tilde \dX)
\]
is an equivalence.
\end{lem}

\begin{proof}
Similarly as before, but using Lemma \ref{lemma:fullyfaithfulnessofjlowerstarpupperstar}, it suffices to show that 
\[
\Hom(j_{*}\sO_{\dE}\otimes \sO_{\tilde \dX}(-k), j_{*}\sO_{\dE} \otimes \sO_{\tilde \dX}(-l)) \simeq 0
\]
and 
\[
\Hom(\sO_{\tilde \dX}, j_{*}\sO_{\dE} \otimes \sO_{\tilde \dX}(-l)) \simeq 0
\]
for $k = 1, \dots, l-1$.
By the projection formula \cite[Rem.~3.4.2.6]{sag} we have $j_{*}\sO_{\dE} \otimes \sO_{\tilde \dX}(-l) \simeq j_{*}\sO_{\dE}(-l)$ and hence
\begin{align*}
&\Hom(j_{*}\sO_{\dE}\otimes \sO_{\tilde \dX}(-k), j_{*}\sO_{\dE} \otimes \sO_{\tilde \dX}(-l)) \\
&\simeq \Hom(j_{*}\sO_{\dE}\otimes \sO_{\tilde \dX}(-k), j_{*}\sO_{\dE}(-l)) \\
&\simeq \Hom(j^{*}j_{*}\sO_{\dE}\otimes \sO_{\dE}(-k), \sO_{\dE}(-l))  \\
&\simeq \Hom(j^{*}j_{*}\sO_{\dE}, \sO_{\dE}(k-l))  \\
&\simeq \Hom(\sO_{\dE}, \sO_{\dE}(k-l)) \oplus \Hom(\sO_{\dE}, \sO_{\dE}(k-l-1)[-1]) \\
&\simeq 0
\end{align*}
by \eqref{eq:jupperstarjlowerstar} and Lemma~\ref{lem:vanishing}.

The second vanishing is obtained in the same way.
\end{proof}

\begin{lem}\label{lem:subquotients-equivalent}
The functor $j^{*}$ induces equivalences $\Perf^{l}(\tilde \dX)/\Perf^{l-1}(\tilde \dX) \xrightarrow{\simeq} \Perf^{l}(\dE)/\Perf^{l-1}(\dE)$.
\end{lem}
\begin{proof}
By Lemmas \ref{lem:technical-lemma-1} and \ref{lem:technical-lemma-2} it 
 is enough to prove that the composite
\begin{multline*}
\Perf(\dY) \xrightarrow{j_{*}q^{*}(-)\otimes \sO_{\tilde \dX}(-l)} \Perf^{l}(\tilde \dX) \to \Perf^{l}(\tilde \dX)/\Perf^{l-1}(\tilde \dX) \xrightarrow{j^{*}}\\\to \Perf^{l}(\dE)/\Perf^{l-1}(\dE)
\end{multline*}
is equivalent to
\[
\Perf(\dY) \xrightarrow{q^{*} \otimes \sO_{\dE}(-l)} \Perf^{l}(\dE) \to \Perf^{l}(\dE)/\Perf^{l-1}(\dE),
\]
or more precisely that for $F\in \Perf(\dY)$ the natural map
\[
j^{*}(j_{*}q^{*}F \otimes \sO_{\tilde \dX}(-l)) \to q^{*}F \otimes \sO_{\dE}(-l)
\]
has its cone in $\Perf^{l-1}(\dE)$. It suffices to check this for the generator $F=\sO_{Y}$.
But since $j^{*}j_{*}\sO_{E} \to \sO_{E}$  has cone $\sO_{E}(1)[1]$,
the map in question has cone $\sO_{E}(-l+1)[1]$ and we are done.
\end{proof}
 
We are finally in the position to {\it prove Theorem~\ref{thm:descent-for-derived-blow-up}}.
Consider the diagram
\[
 \xymatrix@C=4ex{
 \Perf(X) \ar[d]^{i^{*}} \ar[r]^-{p^{*}}_-{\simeq} & \Perf^{0}(\tilde \dX) \ar[d]^-{j^{*}}   \ar@{}[r]|\subseteq  &\Perf^{1}(\tilde \dX) \ar[d]^-{j^{*}}  \ar@{}[r]|\subseteq  & \cdots \ar@{}[r]|\subseteq       &\Perf^{r}(\tilde \dX) \ar[d]^-{j^{*}} \ar@{}[r]|= & \Perf(\tilde \dX)      \\
 \Perf(\dY) \ar[r]^-{q^{*}}_-{\simeq} & \Perf^{0}(\dE) \ar@{}[r]|\subseteq  & \Perf^{1}(\dE) \ar@{}[r]|\subseteq  & \cdots \ar@{}[r]|\subseteq  & \Perf^{r}(\dE)\ar@{}[r]|=& \Perf(\dE).
 }
 \]
By Lemma~\ref{lem:subquotients-equivalent} the functor $j^{*}$ induces equivalences on the successive subquotients. 
By Theorem~\ref{thm:K-is-localizing} $K$-theory sends each square in the above diagram to a homotopy cartesian square of spectra.
Putting these together we get the desired result. \qed

\subsection{Projective bundle formula and Bass fundamental theorem}

For a derived scheme $\dX$ we define the relative projective space over $\dX$ as the pullback $\P^{r}_{\dX} = \dX \times^{\dho}_{\Spec(\Z)} \P^{r}_{\Z}$ in $\dSch$ and  denote the projection $\P^{r}_{\dX}\to \dX$ by $p$. We write $\sO(1)=\sO_{\P^{r}_{\dX}}(1) \in \Perf(\P^{r}_{\dX})$ for the  twisting sheaf given via pullback from $\P^{n}_{\Z}$.
\begin{thm}\label{thm:projective-bundle-formula}
Let $\dX$ be a qcqs derived scheme. Then the functors $p^{*}(-)\otimes\sO(-l)\colon \Perf(\dX) \to \Perf(\P^{r}_{\dX})$ induce an equivalence 
\[
\bigoplus_{l=0}^{r} K(\dX) \xrightarrow{\simeq} K(\P^{r}_{\dX}).
\]
\end{thm}
\begin{proof}
By Zariski descent (Theorem~\ref{thm:Zariski-descent}) we may assume that $\dX$ is affine. Essentially as in Lemma~\ref{lem:explicit-generators}(ii) we then get that $\Perf(\P^{r}_{\dX})$ is generated by $\sO_{\P^{r}_{\dX}}(-l)$ for $l=0, \dots, r$. We define a filtration of $\Perf(\P^{r}_{\dX})$  as above: $\Perf^{l}(\P^{r}_{\dX})$ is generated by $\sO_{\P^{r}_{\dX}}(-k), k=0, \dots, l$. Lemmas \ref{lem:fully-faithful}, \ref{lem:vanishing}, and \ref{lem:technical-lemma-1} show mutatis mutandis   that $p^{*}(-)\otimes\sO(-l)$ induces an equivalence $\Perf(\dX) \to  \Perf^{l}(\P^{r}_{\dX})/\Perf^{l-1}(\P^{r}_{\dX})$.
It follows that the fibre sequence of $K$-theory spectra associated to the exact sequence
\[
\Perf^{l-1}(\P^{r}_{\dX}) \to \Perf^{l}(\P^{r}_{\dX}) \to \Perf^{l}(\P^{r}_{\dX})/\Perf^{l-1}(\P^{r}_{\dX})
\]
splits. Inductively, this implies the theorem.
\end{proof}

For a derived scheme $\dX$ we write $\dX[t]= \dX \times_{\Spec(\Z)}^{\dho} \Spec(\Z[t])$, etc.
\begin{thm}\label{thm:Bass-fundamental}
Let $\dX$ be a qcqs derived scheme. For every integer $n$ there is an exact sequence
\[
0 \to K_{n}(\dX) \to K_{n}(\dX[t]) \oplus K_{n}(\dX[t^{-1}]) \to K_{n}(\dX[t,t^{-1}]) \to K_{n-1}(\dX) \to 0.
\]
\end{thm}
\begin{proof}
The usual proof using Zariski descent (Theorem~\ref{thm:Zariski-descent}) and the projective bundle formula for $\P^{1}_{\dX}$ show this.
\end{proof}

\section{Pro-excision for simplicial rings}

\label{sec:ProExcision}

\subsection{Pro-systems}\label{ex.subsec1}
In this subsection we summarize some properties and notations of pro-systems.
Given a category $\mathbf C$ one constructs the corresponding category of pro-systems
in $\mathbf C$~\cite[Sec.~2]{Isak01}. Any small cofiltered category $I$ together with a functor $X\colon I\to
\bf C$, written $n\mapsto X_n$, gives rise to a pro-system in $\mathbf{C}$, which, following
Deligne, we denote by $\prolim{n} X_n$. All pro-systems we work with in this paper are indexed by a
countable category~$I$. 
If the pro-systems $X = \prolim{n\in I} X_{n}$ and $Y = \prolim{n\in I} Y_{n}$ are indexed by the same category $I$, then by a level map $X \to Y$ we mean  a natural transformation between these $I$-shaped diagrams.

Let $f\colon X\to Y$ be a morphism of pro-simplicial sets, say
\[
X=\prolim{n\in I} X_n \quad \text{and}\quad
Y= \prolim{n\in I} Y_n.
\]
The homotopy pro-groups and integral homology pro-groups  of $X$ are defined as 
\[
\pi_i(X,x_0)= \prolim{n} \pi_i( X_n,(x_0)_n)  \quad \text{and}\quad H_i(X)= \prolim{n} H_i(X_n ,\Z).
\]
Here $x_0: \{ *\} \to X$ is a morphism from the constant pro-system of a one point set to
$X$. In the following we usually omit the base point in the notation of homotopy groups.

 There is a general definition for $f$
to be a weak equivalence (we shall simply speak of equivalence), see
\cite[Def.~6.1]{Isak01}.   We need this notion only in two special cases.

\begin{itemize}
\item[(i)] Assume that for all $n\in I$ the spaces $X_n$ and $Y_n$ are pointed and connected
  and all maps preserve the base points. Then $f$ is an equivalence if and only if for all
  $i>0$ the map
\begin{equation}\label{ex.eq.pro1}
f_* \colon \pi_i (X) \to \pi_i (Y)
\end{equation}
is an isomorphism of pro-groups, see \cite[Cor.~7.5]{Isak01}.
\item[(ii)] Assume that $f$ is induced by a morphism of pro-systems of group-like simplicial monoids.  Then $f$ is an equivalence if and only if for all
  $i\ge 0$ the map~\eqref{ex.eq.pro1} is an isomorphism of pro-groups.
\end{itemize}

In particular an equivalence of pro-simplicial rings is a
morphism $f\colon A\to B$ inducing isomorphisms of pro-groups $\pi_i(A)\to \pi_i(B)$ for all
$i\ge 0$. Here $A$ and $B$ are pointed by $0$.

\begin{lem}\label{ex.lem.basechiso}
Let $A=(A_n)$, $B=(B_n)$ and $C=(C_n)$ be pro-systems of simplicial rings.
Given level morphisms $A\to B$ and  $A\to C$ such that $A\to B$ is an equivalence of
pro-simplicial rings, then the base change $f_C\colon C\to B \otimes^L_A C$ formed levelwise is an
equivalence of pro-simplicial rings.
\end{lem}

\begin{proof}
By \cite[Thm.~II.6.6]{Quillen67} there is a convergent spectral sequence
\[
\prolim{n} \Tor^{\pi_* A_n}( \pi_* B_n , \pi_* C_n ) \Rightarrow \prolim{n} \pi_* ( B_n \otimes^L_{A_n} C_n).
\]
Our assumptions imply that 
\[
\prolim{n} \pi_* C_n = \prolim{n} \Tor^{\pi_* A_n}( \pi_* A_n , \pi_* C_n )  \xrightarrow{\cong} \prolim{n} \Tor^{\pi_* A_n}( \pi_* B_n , \pi_* C_n ) 
\]
is an isomorphism of pro-groups in each degree.
\end{proof}

\smallskip

The following lemma is well known, it can be deduced for example from the equivalence of (a) and
(e) in  \cite[Thm.~7.3]{Isak01}.

\begin{lem}\label{ex.prop.homequ}
If $f\colon X\to Y$ is an equivalence of pro-spaces, then the map of pro-systems of integral homology groups
\[
f_* \colon H_i(X) \to H_i(Y)
\]
is an isomorphism.
\end{lem}

A partial converse to Lemma~\ref{ex.prop.homequ} is given by the following pro-version of Whitehead's
theorem. 
\begin{prop}\label{eq.propwhite}
  Let $f\colon X \to Y$ be a morphism of pro-systems of pointed connected nilpotent spaces  
such that
\[
f_* \colon  H_i(X) \to H_i(Y)
\]
is an isomorphism of pro-groups for all $i\ge 0$. Then $f$ is an equivalence of pro-spaces.
\end{prop}
Proposition~\ref{eq.propwhite} follows from a combination of  Corollary~III.5.4.(ii)
and Proposition~III.6.6 from~\cite{BousfieldKan} for pro-systems indexed by a countable category. It seems likely that
it also holds for uncountable pro-systems.

\smallskip

\begin{defn}\label{ex.defn.prospec}
We say that a morphism of pro-spectra $f\colon X\to Y$ is an \emph{equivalence} if 
\[
f_* \colon \prolim{n} \pi_i(X_n) \to \prolim{n} \pi_i(Y_n) 
\]
is an isomorphism of pro-groups for all $i\in \Z$. We say that a commutative square of
pro-spectra
\begin{equation}\label{ex.eq.hocopro}
\begin{split}
\xymatrix{
X' \ar[r]^{f'} \ar[d]_{\psi} & Y' \ar[d]^\phi\\
X\ar[r]^f  & Y
}
\end{split}
\end{equation}
is \emph{homotopy cartesian} if $X'$ is equivalent to a levelwise homotopy limit of the other
part of the square.
\end{defn}

\begin{rmk}\label{ex.rmkcohoca}
The conditions formulated in Definition~\ref{ex.defn.prospec} are weaker than the
corresponding conditions with respect to the model structure of~\cite{Isaksen04}. Our ad hoc
definition provides more flexibility in the arguments of this section. Anyway, all the homotopy cartesian
squares of pro-spectra encountered in Section~\ref{sec:Introduction} and in
Section~\ref{sec.proof} are in fact homotopy cartesian in the stronger sense of \cite{Isaksen04}.
This follows at once from
Lemma~\ref{lemma:nilinvarianceinsmalldegrees}. 
\end{rmk}

Using standard arguments from topology (cf.\ \cite[Sec.~II.8]{GJ}) one shows:

\begin{lem}\label{ex.lem.hocopro} 
The following are equivalent for a commutative square of pro-spectra~\eqref{ex.eq.hocopro}. 
\begin{itemize}
\item[(i)] The square ~\eqref{ex.eq.hocopro} is homotopy cartesian.
\item[(ii)] $\hofib (f')\to \hofib(f)$ is an equivalence.
\end{itemize}
\end{lem}
Here the homotopy fibre is calculated levelwise.
In particular Lemma~\ref{ex.lem.hocopro} implies that the composition of homotopy
cartesian squares of pro-spectra is homotopy cartesian.

Applying \cite[Thm.~4.1]{Isak02} to homotopy pro-groups  yields:
\begin{lem}\label{ex.lem.sephoiso}
Let $X=(X_{n,m})$ and $Y=(Y_{n,m})$ be pro-systems of spectra in the independent indices $n$
and $m$. Let $ f:X\to Y$ be a level map such that for each $m$ the map 
\[
\prolim{n} X_{n,m} \to \prolim{n} Y_{n,m}
\]
is an equivalence. Then
\[
\prolim{n,m} X_{n,m} \to \prolim{n,m} Y_{n,m}
\]
is an equivalence.
\end{lem}

\smallskip

If $A$ is a pro-simplicial ring we denote by $K(A)$ the corresponding pro-system of $K$-theory spectra
and by $K_i(A)$ the pro-system of $i$-th homotopy groups.

\begin{prop}\label{ex.prop.isokth}
If $f\colon A\to B $ is an equivalence of pro-simplicial rings, then 
$f_*\colon K(A)\to K(B)$ is an equivalence of pro-spectra. 
\end{prop}

\begin{proof}
We check that $f_*$ induces an  isomorphism of pro-homotopy groups $K_i$ for all $i\in \Z$.
For $i\le 1$ this is clear from Theorem~\ref{thm:identificationofthenegativekgroups}. For $i\ge 1$ we can calculate  $K_i$ in
terms of the plus-construction by Proposition~\ref{prop:plusequalswaldhausen}, so we have to show that
\[
\BGLh (A)^+ \to \BGLh (B)^+
\]
is an equivalence of pro-spaces. Here the plus construction is applied
levelwise.  Note that these are pro-systems of connected loop
spaces by Proposition~\ref{prop:plusequalswaldhausen}, so by Proposition~\ref{eq.propwhite} it is sufficient to show that the upper
horizontal map in
\begin{equation}\label{ex.eq.pro2}
\begin{split}
\xymatrix{
H_i(\BGLh (A)^+) \ar[r] & H_i(\BGLh (B)^+) \\
 H_i(\BGLh (A))  \ar[u]^{\wr}  \ar[r]   & H_i(\BGLh (B))  \ar[u]^{\wr}
}
\end{split}
\end{equation}
is an isomorphism of pro-groups for all $i\ge 0$. Note that the vertical maps are
levelwise isomorphisms. As 
\[
\pi_i(\BGLh (A))=   \begin{cases} 
   \GL(\pi_0(A))  & \text{if } i=1 \\
   \rM(\pi_{i-1}(A))       & \text{if } i>1
  \end{cases} 
\]
by Lemma~\ref{lemma:segalstheoremappliedtobglhat}
and correspondingly for $B$, the map
\[
\BGLh (A) \to \BGLh (B)
\]
is an equivalence of pro-spaces. So by Lemma~\ref{ex.prop.homequ} the lower
horizontal map in~\eqref{ex.eq.pro2} is an isomorphism of pro-groups.
\end{proof}

\subsection{Statement of results}

Let $R$ be a commutative ring an let $\mathbf a= (a_1, \ldots , a_r)\in R^r$. We write
$\mathbf a(n)$ for the sequence $(a_1^n, \ldots ,  a_r^n)$. For an $R$-module $M$ we denote
by $\mathbf a(n)M$ the submodule  $a_1^nM + \cdots + a_r^n M$ of $M$ and we denote by
$M/\mathbf a (n)$ the corresponding quotient module.
To simplify the notation for any simplicial ring $A$ over $R$ we write $A(n)$ for
$\Ko(A;\mathbf a(n)) $ (see Definition~\ref{dfn:derived-pullback}).

The following lemma reflects a well-known algebraic observation. It explains
why pro-systems help to bridge the gap between derived geometry and classical
geometry.

\begin{lem}\label{ex.koslem}
If $R$ is noetherian and $M$ is a finitely generated $R$-module, we get the vanishing
\[
\prolim{n} H_i (M; \mathbf a(n))=0
\]
of Koszul homology for all $i> 0$.
\end{lem}

\begin{proof}
We prove the lemma by induction on the length $r$ of the sequence $\ba = (a_1, \ldots , a_r)$. In case $r=0$ it is clear. Assume it is known for
sequences of length less than $r>0$.
Set $\mathbf a'(n)=(a_1^n, \ldots , a_{r-1}^n)$. Then we get an exact sequence
\[
\cdots \xrightarrow{a_r^n} H_i(M; \mathbf a'(n)) \to H_i(M; \mathbf a(n)) \to  H_{i-1}(M;
\mathbf a'(n)) \xrightarrow{a_r^n} \cdots .
\]
Our induction assumption tells us that
\[
\prolim{n} H_i (M; \mathbf a'(n))=0
\]
for $i>0$, so it is sufficient to show that the pro-group
\[
m\mapsto K(m)= \ker[ H_{0}(M; \mathbf a'(n)) \xrightarrow{a_r^m} H_{0}(M; \mathbf a'(n))  ]
\]
is trivial for every $n>0$. Note that $H_{0}(M; \mathbf a'(n)) =M/\mathbf a'(n)$ is a
finitely generated $R$-module and that the transition map of the pro-system $K(m+1)\to K(m)$ is
multiplication by $a_r$. By the noetherian hypothesis there exists $m_0>0$, depending
on $n$, such that the
natural inclusion $K(m_0)\subseteq K(m_0+1)$ is an isomorphism. Then the transition
map $K(m+m_0) \to K(m)$ vanishes for all $m>0$. 
\end{proof}

\begin{lem}\label{ex.lemmodid} For a noetherian simplicial ring $A$ over $R$
the canonical morphism $A\to A(n)$ induces isomorphisms of pro-groups
\[
\prolim{n} (\pi_i A )/ \mathbf a (n)  \xrightarrow{\cong} \prolim{n} \pi_i A(n)
\]
for all $i\ge 0$.
\end{lem}

\begin{proof}
Use the spectral sequence
\[
E^2_{p q}(n) = H_p( \pi_q(A) ;\mathbf a (n)) \Rightarrow  \pi_{p+q} (  A(n)  )
\]
and Lemma~\ref{ex.koslem} with the ring $\pi_0(A)$ and the module $\pi_q(A)$.
\end{proof}

Let $(A_m)_{m}$ and $(B_m)_{m}$ be  pro-systems of commutative
simplicial rings over $R$.

\begin{thm}[Pro-excision for simplicial rings]\label{ex.mainthm}
Consider a morphism of pro-systems of noetherian simplicial rings $\phi\colon (A_m)_{m}\to
(B_m)_{m}$ over $R$. Assume that $\phi$ induces an isomorphism 
\[
\prolim{n,m} \mathbf a (n) \pi_i A_m \xrightarrow{\cong}  \prolim{n,m} \mathbf a (n) \pi_i B_m
\]
of pro-groups  for all $i\ge 0$. Then  
\[
\xymatrix{
\prolim{m} K(A_m) \ar[r] \ar[d] &  \prolim{n,m} K(A_m(n))  \ar[d]\\
\prolim{m} K(B_m) \ar[r] &  \prolim{n,m} K( B_m(n)) 
}
\]
is homotopy cartesian.
\end{thm}
For the notion of homotopy cartesian squares of pro-spectra see Definition~\ref{ex.defn.prospec}.
Actually, we will apply the theorem only for constant pro-systems $(A_m)$ and $(B_m)$,
but the more general formulation is needed to prove the theorem by induction on $r$ in Subsection~\ref{ex.sec.proof}.

Morrow \cite{Morrow14} and  Geisser--Hesselholt \cite{GeisserHesselholt},  \cite[Thm.~3.1]{GH11} have shown the following analog of Theorem~\ref{ex.mainthm} for discrete rings,
based on the work of Suslin and Wodzicki \cite{SuslinWod}, \cite{Suslin}. Our
approach to Corollary~\ref{ex.corex1}   simplifies their proof of this discrete result, since we do not use any Tor-unitality.

\begin{cor}[Pro-excision for rings]\label{ex.corex1}
Let $\phi\colon A\to B$ be a homomorphism of noetherian commutative rings.  Let $I\subseteq A$ be an
ideal which $\phi$ maps isomorphically onto an ideal $J\subseteq B$. Then 
\[
\prolim{n} K(A,I^n) \xrightarrow{\simeq} \prolim{n} K(B,J^n)
\]
is an equivalence.
\end{cor}

Another consequence of Theorem~\ref{ex.mainthm} that we will use in
Section~\ref{sec.proof} is the following:

\begin{cor}\label{ex.corex2}
Assume that $\dX=\Spec(A)$ is a noetherian derived scheme over the ring $R$ such that
the open complement of $\dV_\dX ( \ba )$ is an ordinary scheme. Then
\[
\prolim{n} K( \dX, \dV_\dX( \mathbf a (n) ) ) \to \prolim{n} K( t\dX, t \dV_\dX( \mathbf a (n) ))
\]
is an equivalence.
\end{cor}

Here $\dV_\dX ( \mathbf a (n) )= \dV (\mathbf a (n) ) \times^h_{\Spec R} \Spec A$. In order to prove
Corollary~\ref{ex.corex2} apply Theorem~\ref{ex.mainthm} to the morphism $A\to \pi_0 A$
and observe that 
the assumption implies that $\prolim{n} \ba(n)\pi_{i}(A) \cong 0$ for $i>0$ and that 
from Lemma~\ref{ex.lemmodid} and from Proposition~\ref{ex.prop.isokth} we
get an equivalence
\[
 \prolim{n}  K(   \dV_{t \dX}( \mathbf a (n) ))  \xrightarrow{\simeq}  \prolim{n}  K(  t \dV_\dX( \mathbf a (n) )).
\]

\subsection{Proof of pro-excision} \label{ex.sec.proof}

We now prove  Theorem~\ref{ex.mainthm}.
Using an induction on $r$ we reduce it to the following proposition
from homology theory. Its technical proof is very similar to the proof of Suslin's \cite[Thm.~3.5]{Suslin}. However there are some subtleties as we have to work with group-like simplicial
monoids instead of groups. In order to provide a convincing argument we need to reproduce
parts of Suslin's work here and we need to check that
everything works in our framework. Of course we do not claim much originality.

\begin{prop}\label{ex.prop.trivac}
  Let $I$ be a commutative simplicial ring which is not necessarily unital. Consider an
  element $a\in I_0$ such that the multiplication map
  $I\xrightarrow{ a} I$ is injective in each degree. Then $\GL (\mathbb Z)$ acts trivially on the image of
\[
 H_i(\GLh(aI)) \to H_i(\GLh(I))  ) 
\]
for all $i\ge 0$.
\end{prop}

Here the  homology $H_i(G)$ of a simplicial monoid $G$ is defined as the integral homology
of the classifying space of $G$. For more details about monoid homology see Appendix~\ref{ex.sec.hom}.

\begin{proof}[Proof of Proposition~\ref{ex.prop.trivac}]
Fix $a\in I$ with the property formulated in Proposition~\ref{ex.prop.trivac}.
Arguing as in \cite[Prop.~1.5]{SuslinWod}, we are reduced 
to show that the relative  homology map
\[
H_i(\GLh(aI) \ltimes \rM_{\infty,1}(aI), \GLh(aI)) \to H_i(\GLh(I) \ltimes \rM_{\infty,1}(I), \GLh(I))
\]
vanishes for all $i\ge 0$.
Here $\rM_{\infty,1}$ denotes $(\infty\times 1)$-matrices. 
This map factors through
\begin{equation}\label{ex.eq3}
H_i(\GLh(I) \ltimes \rM_{\infty,1}(I), \GLh(I)) \xrightarrow{({\rm id},a)} H_i(\GLh(I) \ltimes \rM_{\infty,1}(I), \GLh(I)) ,
\end{equation}
induced by the identity on $\GLh(I)$ and by multiplication by $a$ on $\rM_{\infty,1}(I)$.
For simplicity of notation we denote by $a$ a map which is multiplication by $a$ on the
$\rM$-part and the obvious inclusion on the other parts.
We will show that~\eqref{ex.eq3} vanishes.  This follows from:

\begin{claim}\label{ex.cl1}
For any $n,m>0$ and $i\le m$ the map
\begin{multline}\label{ex.eq4}
H_i(\GLh_n(I) \ltimes \rM_{n,1}(I), \GLh_n(I)) \xrightarrow{a}\\\to H_i(\GLh_{n+m}(I) \ltimes \rM_{n+m,1}(I), \GLh_{n+m}(I))
\end{multline}
vanishes.
\end{claim}

Set
\[
\rG_{n,m} (I) = \begin{pmatrix} \GLh_n(I) & \rM_{n,m}(I) \\ 0 & \rT_m(I)    \end{pmatrix},
\]
where we write $\rT_m(I)$ for the upper triangular matrices. 
In order to show Claim~\ref{ex.cl1} we consider the commutative diagram of morphisms of simplicial
monoids
\[
\xymatrix@C=4ex{
\GLh_n(I) \ar[d]  \ar[r] &  \GLh_n(I) \ar[r] \ar[d] &  \GLh_{n+m}(I) \ar[d] \\
\GLh_n(I) \ltimes \rM_{n,1}(I) \ar[r]^-{a} & \rG_{n,m}(I) \ltimes \rM_{n+m,1}(I) \ar[r] & \GLh_{n+m}(I) \ltimes \rM_{n+m,1}(I)
}
\]
and the relative homology of the vertical maps.
So it is sufficient to see that 
\[
H_i(\GLh_n(I) \ltimes \rM_{n,1}(I), \GLh_n(I)) \xrightarrow{a} H_i( \rG_{n,m}(I) \ltimes \rM_{n+m,1}(I), \GLh_n(I)) 
\]
vanishes for $i\le m$, 

We write $\rT_{n,m}(I)$ for the kernel of the projection  to the upper left block
$\rG_{n,m}(I) \to \GLh_n(I)$, so $\rG_{n,m}(I)=\GLh_n(I) \ltimes \rT_{n,m}(I)$.
In view of Lemma~\ref{lem.vanreho} from the Appendix, the proof of Proposition~\ref{ex.prop.trivac} is finished by applying Lemma~\ref{ex.cll2} below.
\end{proof}

\begin{lem}\label{ex.cll2}
The $\rM_{n,n}(I)$-equivariant morphism
\[
\varphi \colon \rM_{n,1}(I) \xrightarrow{a} \rT_{n,m}(I) \ltimes \rM_{n+m,1}(I)
\]
is homologically $m$-constant. More precisely, there is an $m$-homotopy as in
Definition~\ref{ex.def.homoto} which is $\rM_{n,n}(I)$-equivariant and functorial in the pair $(I,a)$.
\end{lem}

\begin{proof}
As our construction will be functorial in the pair $(I,a)$ we can assume without loss of
generality that $I$ is a (non-unital) discrete ring.
We use induction on $m\ge 0$, the case $m=0$ being clear. 

Consider now $m>0$ and assume the lemma has been shown for smaller $m$. 
Let $\psi\colon \rM_{n,1}(I)\to \rT_{n,m}(I)$ be the homomorphism which puts an $n$-vector into
the first $n$-entries of the
last column of the identity matrix. Set $u=(0, \ldots , 0, a)^t\in \rM_{n+m,1}(I)$. 
We have the
equality 
\[
u^{-1}\psi(v)u = \psi(v) \cdot \varphi(v)\quad \text{ for } v\in \rM_{n,1}(I).
\]
Hence Construction~\ref{ex.const.conj} gives  a
\begin{itemize}
\item[(i)]\label{item:someitemi}  homotopy between $C(\psi)$ and $C(\psi\cdot\phi)$ 
\end{itemize}
which is functorial in the pair $(I,a)$ and in particular $\rM_{n,n}(I)$-equivariant.

As the images of $\varphi$ and $\psi$ commute and as $\varphi$ factors through
\[
 \rM_{n,1}(I) \xrightarrow{a} \rT_{n,m-1}(I) \ltimes \rM_{n+m-1,1}(I),
\]
which by our induction assumption is homologically $(m-1)$-constant,  Construction~\ref{const.eilzil} gives
a functorial and hence $\rM_{n,n}(I)$-equivariant
\begin{itemize}
\item[(ii)]\label{item:someitemii} $m$-homotopy between
$C(\psi \cdot \varphi ) $ 
and $C(\psi) + C(\varphi)$.
\end{itemize}
Putting the homotopies from (i)
and (ii) together we see that $\varphi$ is homologically $m$-constant in a functorial and
$\rM_{n,n}(I)$-equivariant way.
\end{proof}

\medskip

Before we give the proof of Theorem~\ref{ex.mainthm},  let us introduce some notation.
Consider a pro-system of abelian groups $(F_n)_{n\in I}$. When there is a group $G$ acting
compatibly on all $F_n$ we say that this action is \emph{pro-trivial} if for each $n\in I$ there
exists $m\to n$ in $I$  such that $G$ acts trivially on the image of $F_m\to F_n$. 

Now we begin with the \emph{proof of Theorem~\ref{ex.mainthm}} using an induction on $r$.

\medskip

\noindent {\it Base case $r=1$.}
From Theorem~\ref{thm:identificationofthenegativekgroups} we know  that for $i\le 1$ and any simplicial ring $C$ the natural map
\[
K_i(C) \to K_i (\pi_0 C)
\]
is an isomorphism. This, together with excision for non-positive $K$-theory of
discrete rings
\cite[Thm.~III.4.3]{Weib13}, implies that
\[
\prolim{n,m} K_i(A_m, A_m(n)) \xrightarrow{\cong}   \prolim{n,m} K_i(B_m, B_m(n)) 
\]
is an isomorphism for $i\le 0$.
Hence, in order to prove Theorem~\ref{ex.mainthm}, we can assume that  $i>0$.

By replacing $A_m$ and $B_m$ by equivalent simplicial rings we can assume without loss of
generality that $a=a_1$ is a non-zero divisor on all of them. Then the canonical maps
$A_m(n)\to A_m/a^n A_m$ and $B_m(n)\to B_m/a^n B_m$ are equivalences. 
 
 For simplicity of notation
we suppress the index $m$ in the following. If $m$ is not mentioned then the result  
holds for each $m>0$ separately.

Consider the two homotopy fibres
\begin{align*}
\tilde F(n) &= \hofib ( \BGLh (A)^+ \to \BGLh'(A/a^nA)^+ ), \\
\tilde G(n) &= \hofib ( \BGLh (B)^+ \to \BGLh'(B/a^nB)^+ ). 
\end{align*}
Here $\GLh'(A/a^nA) \subseteq \GLh(A/a^nA)$ is the preimage of \[ \im (\GL(\pi_0(A) ) \to
\GL(\pi_0(A/a^n A)) )\]   and correspondingly
for $B/a^nB$. Then $\BGLh'(A/a^nA)^+$ is equivalent to the connected covering of
$\BGLh(A/a^nA)^+$ corresponding to the image of $\GL (\pi_0 A)$ in $\pi_1 \BGLh(A/a^nA)^+$
and correspondingly for $B/a^nB$.
This means that
\[
\pi_i \tilde F(n) = K_i(A,A(n))\quad \text{ and } \quad \pi_i \tilde G(n) = K_i(B,B(n))
\]
for $i\ge 1$. Combining this observation with Proposition~\ref{eq.propwhite} we see that
it suffices to prove:

\begin{claim}\label{ex.cl.whitiso}
For all $i\ge 0$ the map
\[
\prolim{n,m} H_i (\tilde F(n)) \to \prolim{n,m} H_i (\tilde G(n))
\]
is an isomorphism of pro-groups.
\end{claim}

Now consider the homotopy fibres
\begin{align*}
 F_\ell(n) &= \hofib ( \BGLh_\ell (A) \to \BGLh'_\ell(A/a^nA) ), \\
 G_\ell(n) &= \hofib ( \BGLh_\ell (B) \to \BGLh'_\ell(B/a^nB) ), 
\end{align*}
and set $F(n)=F_\infty (n)$, $G(n)=G_\infty(n)$.  Note that by Lemma~\ref{ex.lemmodid} and
Lemma~\ref{lemma:segalstheoremappliedtobglhat} we
get isomorphisms of pro-groups
\begin{align}\label{ex.eq10a}
\prolim{n} \GLh (a^n \pi_{i-1}(A)) &\cong \prolim{n} \pi_i F (n), \\ 
\label{ex.eq10b}  \prolim{n} \GLh ( a^n \pi_{i-1}(B)) &\cong \prolim{n} \pi_i G (n),
\end{align}
where $\GLh$ stands for $\GL$ if $i=1$ and for $\rM$ if $i>1$.

In particular 
\begin{equation}\label{ex.eq11}
\prolim{n,m} F(n) \xrightarrow{\simeq} \prolim{n,m} G(n)
\end{equation}
is an equivalence of pro-spaces.

We reduce Claim~\ref{ex.cl.whitiso} with the help of \eqref{ex.eq11} and Lemma~\ref{ex.prop.homequ} to:

\begin{claim}\label{ex.cl.withplus}
For each $i\ge 0$ the maps
\begin{align*}
\prolim{n} H_i(F(n) ) \to \prolim{n} H_i (\tilde F(n)),\\
\prolim{n} H_i(G(n) ) \to \prolim{n} H_i (\tilde G(n))\phantom{,} 
\end{align*}
are isomorphisms of pro-groups.
\end{claim}

Note that
there is a canonical action of $\pi_1\BGL (A)$ on $F(n)$ up to homotopy. Using a standard
comparison of Serre spectral sequences, see \cite[Cor.~1.7]{SuslinWod} or \cite[Lem.~1.2]{GeisserHesselholt}, we observe that 
Claim~\ref{ex.cl.withplus} follows from:

\begin{claim}\label{ex.cl.triac}
The group $\pi_1\BGLh (A)$ respectively $\pi_1\BGLh (B)$   acts pro-trivially on
$\prolim{n} H_i(F(n) )$ respectively $\prolim{n} H_i(G(n) )$   for all
$i\ge 0$.
\end{claim}

We restrict to $A$ in the proof of Claim~\ref{ex.cl.triac}.
The map
\[
\prolim{n}\BGLh (a^nA) \xrightarrow{\simeq}\prolim{n} F(n) 
\]
is an equivalence by~\eqref{ex.eq10a} and Lemma~\ref{lemma:segalstheoremappliedtobglhat}. Combining this observation with
Proposition~\ref{ex.prop.trivac} we see that
\begin{equation}\label{ex.eq.act1}
\GL(\Z) \text{ acts pro-trivially on } \prolim{n} H_i(F(n) ),
\end{equation}
where the action is by conjugation.

Suppose $n>0$ and $i\ge 0$ are fixed for the moment. From~\eqref{ex.eq.act1} we deduce that there exists $n'\ge n$ such that the
conjugation action of $\GL(\Z)$ on the image of
\[
H_i(F(n')) \to H_i(F(n))
\]
is trivial.

Consider the canonical morphism 
\[
\iota: F_\ell(n')\times F_\ell(n')\to F(n), \quad \quad(\ell>0)
\]
which is induced by the map of matrices 
\[
(g_1,g_2)\mapsto \begin{pmatrix} g_1 & 0 &0 \\ 0 & g_2 & 0 \\ 0 & 0 & 1_\infty \end{pmatrix}.
\]
Let \[\sigma \colon F_\ell(n')\times F_\ell(n') \to F_\ell(n')\times F_\ell(n')\]
be the permutation map. By our choice of $n'$ we get the equality
\[
\iota_* = \iota_* \circ \sigma_* \colon H_i(F_\ell(n')\times F_\ell(n')) \to H_i(F(n)).
\]
So $\pi_1 \BGLh_\ell(A)$ acts trivially on \[\iota_* (H_i(F_\ell(n'))\oplus 0 ) = \iota_*
  (0\oplus H_i(F_\ell(n')) ). \]
This shows Claim~\ref{ex.cl.triac} and finishes the proof of the base case $r=1$.

\medskip

\noindent {\it Inductive step for  $r-1\rightsquigarrow r$.}
We suppress the index $m$ in $A_m $ and $B_m$ for simplicity of notation. 
We assume Theorem~\ref{ex.mainthm} is known for sequences with fewer than $r>1$ elements. 
Let $A'(n)$ be $\Ko(A;a_r^n)$ and let $B'(n)$ be $\Ko(B; a_r^n)$. Set $\mathbf a'(n)=
(a_1^n,\ldots , a_{r-1}^n )$.
Note that \[A(n) =\Ko( A'(n);\mathbf a'(n)  ) \quad\text{ and }\quad B(n) =\Ko( B'(n);\mathbf a'(n)  ). \]
From the assumptions of the theorem one easily deduces the isomorphisms
\begin{align*}
\prolim{n,m} a_r^n \pi_i A &\xrightarrow{\cong}  \prolim{n,m} a_r^n \pi_i B, \\
\prolim{n,m} \mathbf a'(n) \pi_i A'(n)  &\xrightarrow{\cong} \prolim{n,m} \mathbf a'(n) \pi_i B'(n).
\end{align*}
The second isomorphism follows from the first and Lemma~\ref{ex.lemmodid}.
By our induction assumption the left and right squares in
\[
\xymatrix{
\prolim{m} K(A) \ar[r] \ar[d] & \prolim{n,m} K(A'(n)) \ar[d] \ar[r] & \prolim{n,m}
K(A(n)) \ar[d] \\
\prolim{m} K(B) \ar[r]  & \prolim{n,m} K(B'(n))  \ar[r] & \prolim{n,m} K(B(n))
}
\]
are homotopy cartesian. So the composite square is also homotopy cartesian by the remark following Lemma~\ref{ex.lem.hocopro}. This finishes
the proof of Theorem~\ref{ex.mainthm}.
\qed

\section{Proof of the main theorem}\label{sec.proof}

In this section we prove Theorem~\ref{intro.main.thm}.

\subsection{Descent along finite morphisms}

Let $X$ be a noetherian scheme and consider an abstract blow-up square as in~\eqref{intro.eq1}. 
For notational convenience we use bi-relative $K$-theory $K(X,Y,\tilde X,E)$ of the
square~\eqref{intro.eq1} in this subsection, which is defined as the homotopy
fibre of the map $K(X,Y)\to K(\tilde X, E)$. Using this notation
Theorem~\ref{intro.main.thm} is equivalent to
\begin{equation}\label{pr.bitri}
\prolim{n} K(X, Y_n, \tilde X, E_n) \simeq * .
\end{equation}

\begin{lem}\label{pr.lemredaff}
Assume given an abstract blow-up square~\eqref{intro.eq1} such that for all affine open
$U\subseteq X$ we have
\begin{equation}\label{pr.bitri2}
\prolim{n} K(U, Y_n\times_X U, \tilde X \times_X U, E_n \times_X U)\simeq  *.
\end{equation}
Then~\eqref{pr.bitri} holds.
\end{lem}

\begin{proof}
Consider the presheaf of spectra $\mathcal F_n$ on $X$ given on an open subset $V\subseteq X$ by
\[
\mathcal F_n(V) =  K(V, Y_n\times_X V, \tilde X \times_X V, E_n \times_X V).
\]

First we show that, under the assumptions of the lemma, \eqref{pr.bitri2} holds for all open
$U\subseteq X$ which are separated but not necessarily affine. In fact for $U\subseteq X$
separated choose a finite affine open
covering $\mathcal V= (V_j)_{j\in \{1, \ldots, r \}} $ of $U$. The  \v{C}ech
spectral (see Theorem~\ref{thm:Zariski-descent})
\begin{equation}\label{pr.cechsp}
E^{pq}_{2} =H^p(\mathcal V , \pi_{-q} \mathcal F_n ) \Rightarrow \pi_{-p-q} \mathcal F_n (U)
\end{equation}
implies $ \prolim{n} \mathcal F_n(U)\simeq  *$.  Note that $E^{pq}_{2}= 0$ for $p>r$,
which guarantees the convergence of the spectral sequence.

In a second step we consider a finite open affine covering $\mathcal U = (U_j)_{j\in \{1,
  \ldots, s \}}$ of $X$ and observe that any intersection $\cap_{j\in  J}U_j$ with
$J\subset \{1,\ldots , s\}$ is separated. Using the analog of \eqref{pr.cechsp} for the
covering $\mathcal U$ we finally deduce that $ \prolim{n} \mathcal F_n(X)\simeq *$.
\end{proof}

\begin{prop}\label{pr.findesc}
Consider an abstract blow-up square~\eqref{intro.eq1} with $\tilde X\to X$ finite. 
For such a square Theorem~\ref{intro.main.thm} holds, i.e.~\eqref{pr.bitri} is an equivalence.
\end{prop}

\begin{proof}
By Lemma~\ref{pr.lemredaff} we can assume without loss of generality that $X$ is
affine. Say $X=\Spec A$, $\tilde X= \Spec B$. Let $I\subseteq A$ be the ideal defining $Y$
and let $\phi\colon A\to B$ be the canonical ring homomorphism.
As for any $f\in I$ we have
$A_f\xrightarrow{\cong} B_f$, we deduce that there exists $N>0$ such that $ \phi (I^{N})
B\subseteq \phi(A)$ and $I^N\cap \ker (\phi)=0$. For the second equality we use the Artin--Rees lemma.

 Factoring $\phi$ as $A\to
A/\ker(\phi) \to B$ we see that we can assume that $\phi$ is injective or surjective.

\smallskip
\noindent {\it Case $\phi$ is surjective.}
As $\phi$ maps the ideal $J=I^N$ isomorphically onto an ideal of $B$,
Corollary~\ref{ex.corex1} shows that we get an  equivalence of pro-spectra
\[
\prolim{n} K(A,J^n) \xrightarrow{\simeq} \prolim{n} K(B,\phi(J)^n).
\]

\smallskip
\noindent {\it Case $\phi$ is injective.}
The ideal $J=\phi^{-1}(  \phi(I^N) B)$ of $A$ maps isomorphically onto the ideal $
\phi(I^N) B$ of $B$. Furthermore the pro-systems of ideals $\prolim{n} J^n$ and
$\prolim{n} I^n$ are equivalent, so we finish the proof as in the first case using Corollary~\ref{ex.corex1}.
\end{proof}

\subsection{Reduction to classical blow-ups}\label{ex.sub.redcl}

In this subsection we reduce the proof of Theorem~\ref{intro.main.thm} to the case
in which $\tilde X$ is the blow-up of $X$ in the closed subscheme $Y$. We use without
further reference classical properties of blow-ups as summarized in~\cite[Sec.~5.1]{GR}.  For the rest of
this subsection we work with the
\begin{quote}
{\bf Assumption} $(\dagger)$: For any noetherian scheme $X$ and any closed subscheme $Y\to
X$, consider the blow-up
$\tilde X= \Bl_Y (X)$ and the exceptional divisor $E=\tilde X \times_X Y$. Then the square
\begin{equation}\label{pr.blowkth}
\begin{split}
\xymatrix{
K(X) \ar[r] \ar[d] & \prolim{n} K(Y_{n}) \ar[d] \\
K(\Bl_{Y} X) \ar[r] & \prolim{n} K(E_n)
}
\end{split}
\end{equation}
is homotopy cartesian, where as usual $Y_n$ and $E_n$ are infinitesimal thickenings.
\end{quote}
Note that by Lemma~\ref{pr.lemredaff} Assumption $(\dagger)$ holds if it holds for all affine noetherian schemes $X$.

Consider a general abstract blow-up square~\eqref{intro.eq1} with $X$ noetherian. We want
to know whether it induces a homotopy cartesian square on $K$-theory pro-spectra as claimed in Theorem~\ref{intro.main.thm}. 
Using Proposition~\ref{pr.findesc} we can  assume without loss of generality that
$X\setminus Y$ is schematically dense in $X$ and in $\tilde X$. Indeed we can replace $X$
and $\tilde X$ by the
schematic closures of $X\setminus Y$ and we can replace $Y$ and $E$ by their pullbacks without changing $\prolim{n} K(X,Y_n)$ and
$\prolim{n}K (\tilde X, E_n)$ up to equivalence.

By Raynaud--Gruson's {\it platification par \'eclatement} \cite[Sec.~5]{GR} as applied in \cite[Lem.~2.1.5]{Temkin}, we can find a closed
subscheme $Y'\to X$ such that $Y'\subseteq Y$ as sets and such that  $\Bl_{Y'} X \to X$
factors through $\tilde X\to X$. 
Denote by $Y'_n$ the $n$-th infinitesimal thickening of $Y'$ in $X$.

\begin{claim}\label{pr.clblre} Assume that $(\dagger)$ holds. Then  the square
\[
\xymatrix{
K(X) \ar[r] \ar[d] & \prolim{n} K(Y_{n}) \ar[d] \\
K(\Bl_{Y'} X) \ar[r] & \prolim{n} K(Y_n \times_X \Bl_{Y'} X)
}
\]
is homotopy cartesian.
\end{claim}

\begin{proof}[Proof of Claim~\ref{pr.clblre}]
Set $E'_n= Y_n \times_X \Bl_{Y'} X$.  Let $E''_n$ be the schematic closure of $Y_n\setminus Y'$ in
$E'_n$, so $E''_n = \Bl_{Y' \cap Y_n} Y_n $, where $\cap$ is the schematic intersection inside $X$.
Here is a picture of the situation.
\[
\xymatrix{
     & E_{n}'' \ar[d]  \ar@{}[dr]|{\boxed{1}}   & E_{n}'' \times_{Y_{n}} (Y_{m}' \cap Y_{n}) \ar[d]  \ar[l] \\
\Bl_{Y'}X \ar[d]   \ar@{}[dr]|{\boxed{2}}    & E_{n}' \ar[l] \ar[d]    \ar@{}[dr]|{\boxed{3}}     & E_{n}' \times_{Y_{n}} (Y_{m}' \cap Y_{n}) \ar[d]  \ar[l]\\
X & Y_{n} \ar[l] & Y_{m}' \cap Y_{n}\ar[l]
}
\]
As a pro-system in $n$ and $m$, the composed square $\boxed{2}+\boxed{3}$ induces a homotopy cartesian square of $K$-theory pro-spectra by  assumption $(\dagger)$. 
Similarly, for fixed $n$ the composed square $\boxed{1}+\boxed{3}$  induces a homotopy cartesian square of $K$-theory pro-spectra as pro-system in $m$.  By Lemma~\ref{ex.lem.sephoiso} this remains true as a pro-system in $m$ and $n$.
But $\boxed{1}$ induces a homotopy cartesian square of $K$-theory pro-spectra by Proposition~\ref{pr.findesc}. Hence the same is true for $\boxed{3}$ and then also for $\boxed{2}$ as desired.
\end{proof}

\smallskip

Under our assumption $(\dagger)$ we will prove the following two statements in this order.

\begin{itemize}
\item[(i)] For each $i\in \Z$ and each abstract blow-up square~\eqref{intro.eq1} the
  morphism 
\[
\prolim{n} K_i(X,Y_n) \to \prolim{n} K_i(\tilde X, E_n)
\]
is a monomorphism of pro-groups.
\item[(ii)]
For each $i\in \Z$ and each abstract blow-up square~\eqref{intro.eq1} the
  morphism 
\[
\prolim{n} K_i(X,Y_n) \to \prolim{n} K_i(\tilde X, E_n)
\]
is an isomorphism of pro-groups.
\end{itemize}

\begin{proof}[Proof of (i)]
By what is said above we can assume that there exists a blow-up  $\Bl_{Y'} X \to X$ which
factors through $\tilde X$ and such that $Y'\subseteq Y$. Then by Claim~\ref{pr.clblre} the
composition $\gamma$ of
\begin{equation}\label{pr.compeq}
\prolim{n} K_i(X,Y_n) \xrightarrow{\alpha} \prolim{n} K_i(\tilde X, E_n) \xrightarrow{\beta} \prolim{n} K_i(\Bl_{Y'}X, Y_n
\times_X \Bl_{Y'}X ) 
\end{equation}
is an isomorphism for each $i\in \Z$, so $\alpha$ is a monomorphism.
\end{proof}

\begin{proof}[Proof of (ii)]
Arguing as in (i) we consider the morphisms~\eqref{pr.compeq}. Again the composition $\gamma$  is
an isomorphism and by (i) we know that  $\beta$ is a monomorphism, so $\gamma^{-1}\circ
\beta$ is an inverse to $\alpha$.
\end{proof}

Clearly, (ii) is the same as  Theorem~\ref{intro.main.thm}. So in view of Lemma~\ref{pr.lemredaff} we have
shown: 
\begin{prop}\label{pr.redtoblow}
If the square~\eqref{pr.blowkth} from assumption $(\dagger)$ is homotopy cartesian for all
noetherian affine schemes $X$, then  Theorem~\ref{intro.main.thm} is true.
\end{prop}

\subsection{Reduction to derived blow-ups}\label{subsec.redblowde}
As a preliminary step towards the proof of Theorem~\ref{intro.main.thm} we show in this subsection that we can reduce the proof  to the case
of derived blow-ups. In the next subsection we use a refinement  of this observation involving
a tower of derived blow-ups which will complete proof of Theorem~\ref{intro.main.thm}.

Let $A$ be a discrete noetherian ring, $X=\Spec A$, and fix a sequence $\mathbf a=(a_0,\ldots, a_r)\in A^{r+1}$. Let $\mathbf a(n)$
denote the sequence $(a_0^n, \ldots , a_r^n)$. For a definition of the following derived
schemes see Subsection~\ref{subsec:derived-blow-up-squares}.
Let   $\tilde \dX= \Bld_{\bf a} X$ be the  derived blow-up  in the sequence ${\bf a}$, let $\dY(n)=\dV(\mathbf a(n) ) $ be the
Koszul derived scheme, let $\dD(n)=\tilde \dX \times_X^{h} \dY(n) $ be the thickened derived
exceptional divisor,  and let
$\dE(n)$ be the thickened semi-derived exceptional divisor.
To simplify notation we 
denote by $Y(n)=\Spec A/(\mathbf a (n) )$, $E(n)=D(n)$ and $\tilde X$ the corresponding
underlying schemes. 

\begin{table}[b]
\centering
\renewcommand{\arraystretch}{1.3}

\begin{tabular}{| >{\centering}m{1in} |  >{\centering}m{1.4in} |
    >{\centering\arraybackslash}m{2in} |}
\hline  Symbol & Definition & Explanation \\
\hline \hline  $\tilde \dX$ & $\Bld_{\mathbf a } X $ & derived blow-up \\ 
\hline  $\dY(n)$ & $\dV (\mathbf a(n)) $ &  Koszul scheme \\ 
\hline  $\dE(n)$ & $ (Y'(n) \times_{X'} \tilde X') \times^h_{X'} X  $ & (thickened)
semi-derived 
  exceptional divisor\\ 
\hline  $\dD(n)$ & $ \dY(n) \times_X^h \tilde \dX  $ & (thickened) derived
  exceptional divisor\\ 
\hline  primed version & $-$ &  regular sequence, only
                                               $\dD'(n)$ is derived \\
\hline
\end{tabular}\smallskip

\caption{Summary of notation in Subsections \ref{subsec.redblowde} and \ref{subsec.tower}}
\label{pr.table1}
\end{table}

Recall that except for $\dD(n)$ all these derived schemes result from
a derived base change of ordinary schemes along a ring homomorphism $A'\to A$, where $A'$ is a ring endowed with a
regular sequence $\mathbf a'$ mapping to $\mathbf a$. These corresponding derived schemes over
$A'$ will be denoted with a prime; they will be used only in the proof of
Lemma~\ref{pr.lem.coclde} below. See Table~\ref{pr.table1} for a summary.
It might help the reader to revisit the diagram~\eqref{diag:helpfulsquare}.

Note that we get a commutative diagram of derived schemes 
\[
\xymatrix{
E(n) \ar[r] \ar[d] &  \dE (n) \ar[r]  \ar[d] &  \dD(n) \ar[r]  \ar[d] & \tilde \dX\phantom{.}  \ar[d]\\
Y(n) \ar[r] &  \dY(n) \ar@{=}[r]  & \dY(n)\ar[r]  & X  .
}
\]
By Remark~\ref{rmk:ordinary-and-derived-blow-up} 
there is a canonical closed immersion $\Bl_{(\mathbf a )
} X\to \tilde X$ which is an isomorphism over $X\setminus Y$. In particular from
Proposition~\ref{pr.findesc} we deduce the homotopy cartesian square
\begin{equation}\label{pr.blclde}
\begin{split}
\xymatrix{
K(\tilde X )  \ar[r]  \ar[d]& \prolim{n} K(E(n)) ) \ar[d] \\
 K(\Bl_{(\mathbf a )} X) \ar[r]  &    \prolim{n}  K(   Y(n)  \times_X  \Bl_{(\mathbf a )} X ). 
}
\end{split}
\end{equation}

In order to relate assumption $(\dagger)$ from Subsection~\ref{ex.sub.redcl} to derived schemes we consider the square
\begin{equation}\label{pr.homcarder}
\begin{split}
\xymatrix{
K(X) \ar[r] \ar[d] & \prolim{n} K(\dY(n)) \ar[d]\phantom{.} \\
K(\tilde \dX) \ar[r] & \prolim{n} K(\dE(n)).
}
\end{split}
\end{equation}

\begin{lem}\label{pr.lem.coclde}
The morphisms
\begin{align}
\prolim{n} K(\dY(n)) &\xrightarrow{\simeq} \prolim{n} K(Y(n)), \tag{i}\\
 \prolim{n} K(\tilde \dX ,\dD(n)) &\xrightarrow{\simeq} \prolim{n} K(\tilde X , D(n)),
                                    \tag{ii} \\
\prolim{n} K(\dD(n)) & \xrightarrow{\simeq} \prolim{n} K(\dE (n) ) \tag{iii}
\end{align}
are equivalences.
\end{lem}

\begin{proof}
  (i) is a direct consequence of Proposition~\ref{ex.prop.isokth} and Lemma~\ref{ex.lemmodid}.
Let us denote by $\dU_l\subseteq \tilde \dX$ and $U'_l=D_+(a'_l)\subseteq \tilde X' $ the
standard affine open subsets for $l\in \{0,\ldots , r\}$. For $L\subseteq  \{0,\ldots , r\}$
write $\dU_L = \cap_{l\in L} \dU_l$ and $U'_L = \cap_{l\in L} U'_l$.

 In order to show (ii) use the \v{C}ech
 spectral sequence (Theorem~\ref{thm:Zariski-descent}) for the covering $(\dU_l)_l$ of $\tilde \dX$. This reduces us to show that
\[
 \prolim{n} K(\dU_L, \dD(n)\cap \dU_L) \xrightarrow{\simeq} \prolim{n} K(U_L, D(n)\cap U_L)
\] 
is an equivalence for all $L\subseteq \{0,\ldots , r\}$, which is a consequence of
Corollary~\ref{ex.corex2}.

Using the \v{C}ech
 spectral sequence again, we see that (iii) follows if we can show that
\begin{equation}\label{pr.compde}
\prolim{n} K(\dD(n)\cap \dU_L)  \xrightarrow{\simeq} \prolim{n} K(\dE (n)\cap \dU_L )
\end{equation}
is an equivalence for all $L\subseteq \{0,\ldots , r\}$. The map 
\[
\dE (n)\cap \dU_L \to
\dD(n)\cap \dU_L 
\]
is the base change along $X\to X'$ of $E'(n) \cap U'_L \to \dD' (n)\cap U'_L$. For varying
$n$ the latter corresponds to a morphism of pro-system of simplicial rings which is an
equivalence by Lemma~\ref{ex.lemmodid}. As derived base change preserves equivalences
of pro-simplicial rings by Lemma~\ref{ex.lem.basechiso}, we deduce from Proposition~\ref{ex.prop.isokth}
that~\eqref{pr.compde} is an equivalence. 
\end{proof}

In view of Lemma~\ref{pr.lem.coclde} and the cartesian square~\eqref{pr.blclde} it would be sufficient to know that the square~\eqref{pr.homcarder} is
homotopy cartesian in order to deduce assumption $(\dagger )$ for affine $X$. Unfortunately, the Thomason type descent Theorem~\ref{thm:descent-for-derived-blow-up} only says that, if
we replace in~\eqref{pr.homcarder} the pro-systems over $n$ by their values at $n=1$,   the resulting square is
homotopy cartesian.   In the next subsection we
explain a trick which allows us to reduce the problem to this case. 

\subsection{A tower of derived blow-ups}
\label{subsec.tower}

Let the notation be as in the previous subsection. 
Observe that the ideal $(\mathbf a' (2) ) $ of $A'$ is a reduction of the ideal $ (\mathbf a'
)^2$ by \cite[Prop.~8.1.5]{HS}. So by Thm.~8.2.1 there we get a finite  morphism 
\begin{equation}\label{pr.blco1}
\Bl_{(\mathbf a' )^2} X' \to \Bl_{(\mathbf a' (2)
  )} X' 
\end{equation}
 over $X'=\Spec A'$. Furthermore, there is a unique isomorphism  over $X'$
\begin{equation}\label{pr.blco2}
 \Bl_{(\mathbf a' )} X' \xrightarrow{\cong}  \Bl_{(\mathbf a'
  )^2} X'.
\end{equation}  
By derived base change of the composition of~\eqref{pr.blco2} and~\eqref{pr.blco1} along the morphism $X\to X'$ 
 we obtain an affine morphism
\begin{equation}\label{pr.findermap} 
\Bld_{\mathbf a } X \to \Bld_{\mathbf a (2)} X 
\end{equation} 
of derived schemes over $X$ whose underlying scheme morphism is finite.

Set $\tilde \dX(m)= \Bld_{\mathbf a (m)} X $, $\dD(m,n)= \dY(n) \times^h_{X} \tilde \dX(m)
$ and \[ \dE(m,n)= (\Bl_{(\mathbf a'(m))} X' \times_{X'} V(\mathbf a'(n)))\times_{X'}^h
  X\] for $n,m>0$. Iterating the construction of~\eqref{pr.findermap} we get
a tower of derived schemes over $X$ whose underlying scheme morphisms are finite
\[
\tilde \dX \to \tilde \dX(2) \to\tilde \dX(2^2) \to \tilde \dX(2^3) \to\cdots \to
\tilde\dX(2^{n}) \to \cdots.
\]
As before by an italic letter we mean the underlying scheme of a given derived scheme,
which is denoted by the corresponding calligraphic letter.

\begin{lem}\label{ex.lem.coclde2}
The morphisms
\begin{align}
 \prolim{n} K( \dY(2^n)) & \xrightarrow{\simeq} \prolim{n} K( Y(2^n) ), \tag{i} \\
  \prolim{n} K( \tilde \dX (2^n), \dD(2^n,2^n)) & \xrightarrow{\simeq}  \prolim{n} K( \tilde
                                                  X (2^n), D(2^n,2^n)), \tag{ii}\\
\prolim{n} K(\dD(2^n,2^n))  & \xrightarrow{\simeq} \prolim{n} K( \dE(2^n, 2^n)) \tag{iii}
\end{align}
are equivalences.
\end{lem}

\begin{proof}
(i) follows directly from Lemma~\ref{pr.lem.coclde}(i).
For every $m\ge 0$ the canonical morphisms
\begin{align}
  \prolim{n} K( \tilde \dX (2^m), \dD(2^m,2^n)) & \xrightarrow{\simeq}  \prolim{n} K( \tilde
                                                  X (2^m), D(2^m,2^n)), \tag{ii'} \\
 \prolim{n} K(\dD(2^m,2^n))  & \xrightarrow{\simeq} \prolim{n} K( \dE(2^m, 2^n)) \tag{iii'}
\end{align}
are equivalences by Lemma~\ref{pr.lem.coclde}.
If we take the limit over $m$ in (ii') and (iii') in the category of pro-spectra we obtain
the equivalences (ii) and (iii) in view of Lemma~\ref{ex.lem.sephoiso}.
\end{proof}

Now we are ready to prove Theorem~\ref{intro.main.thm}:

\begin{proof}[Proof of Theorem~\ref{intro.main.thm}]
By Theorem~\ref{thm:descent-for-derived-blow-up} the square
\[
\xymatrix{
K(X) \ar[r] \ar[d] &  K(\dY(2^n)) \ar[d] \\
K(\tilde \dX(2^n)) \ar[r] &  K(\dE(2^n, 2^n))
}
\]
is homotopy cartesian for all $n\ge 0$. Considering pro-systems in $n$ and using
Lemma~\ref{ex.lem.coclde2} we get  a homotopy cartesian square
\begin{equation}\label{ex.eq.sq44}
\begin{split}
\xymatrix{
K(X) \ar[r] \ar[d] & \prolim{n} K(Y(2^n)) \ar[d] \\
\prolim{n} K(\tilde X(2^n)) \ar[r] &   \prolim{n} K(E(2^n, 2^n)).
}
\end{split}
\end{equation}
By Proposition~\ref{pr.findesc} we obtain for each $m\ge 0$ an equivalence
\begin{equation}\label{ex.eq.sq45}
\prolim{n} K(\tilde X(2^m), E(2^m, 2^n)) \xrightarrow{\simeq} \prolim{n} K(\tilde X, E( 2^n)).
\end{equation}
Combining the homotopy cartesian square~\eqref{ex.eq.sq44} and the inverse limit
of the equivalence~\eqref{ex.eq.sq45} over $m$, we get the
homotopy cartesian square
\[
\xymatrix{
K(X) \ar[r] \ar[d] & \prolim{n} K(Y(2^n)) \ar[d]\phantom{.} \\
K(\tilde X) \ar[r] &   \prolim{n} K(E( 2^n)).
}
\]
Here we have also used  Lemma~\ref{ex.lem.hocopro}  and Lemma~\ref{ex.lem.sephoiso}.
This homotopy cartesian square together with the homotopy cartesian square~\eqref{pr.blclde} yields the homotopy cartesian square
\[
\xymatrix{
K(X) \ar[r] \ar[d] & \prolim{n} K(Y(2^n)) \ar[d] \\
K(\Bl_{(\mathbf a )} X) \ar[r]  & \prolim{n} K( Y(2^n)\times_X \Bl_{(\mathbf a )} X ).
}
\]
Summarizing, we have shown that assumption $(\dagger)$ from Subsection~\ref{ex.sub.redcl} holds over
any affine
base scheme $X$. So by Proposition~\ref{pr.redtoblow} we obtain Theorem~\ref{intro.main.thm}. 
\end{proof}

\section{Applications}\label{sec:application}

\subsection{Weibel's conjecture}\label{subsec:weibel}

In this subsection, we give a proof of Weibel's conjecture, i.e.\ Theorem~\ref{intro.thm.wc}.
In \cite{KerzStrunk16}, the first and the second author showed the analogous result for
homotopy $K$-theory, using Raynaud--Gruson's \emph{platification par \'eclatement} and
Cisinski's result on cdh-descent for $KH$-theory, for which we give a new proof in the
next subsection.
The proof of Theorem~\ref{intro.thm.wc} follows essentially the same argument but with Theorem~\ref{intro.main.thm} in place of cdh-descent for $KH$-theory.

We first need a reduction. With a proof copied verbatim from \cite[Prop.~3]{KerzStrunk16} we obtain:

\begin{prop}\label{prop:reduceweibeltolocal}
Let $E$ be presheaf of spectra on the category of noetherian schemes which satisfies Zariski descent.
Let $X$ be a noetherian scheme of finite Krull dimension $d$. If for every $x\in X$ the stalk $\pi_i(E)_x$ vanishes for all $i<-\dim(\sO_{X,x})$, then $E_i(X)=\pi_i(E(X))$ vanishes for all $i<-d$. 
\end{prop}

We first show that for a noetherian scheme $X$ the algebraic $K$-theory vanishes below the negative of its  dimension:

\begin{proof}[Proof of Theorem~\ref{intro.thm.wc}(i)]
We argue inductively on the dimension $d$ of the noetherian scheme $X$.
In each step, we can restrict to affine noetherian schemes $X$ by the previous Proposition~\ref{prop:reduceweibeltolocal} applied to the presheaf of spectra $K$.
Since negative algebraic $K$-theory of affines is nil-invariant, we  may also assume that $X$ is reduced.

The case $d=0$ is well-known.
Let $d>0$ and assume that the statement is true for all noetherian schemes of dimension less than $d$.

Let $i<-d$ and consider an element $\xi\in K_i(X)$. By applying \cite[Prop.~5]{KerzStrunk16} to the identity $f=\id_{X}$ we find a projective birational morphism $p\colon X'\to X$ such that $p^*(\xi)=0\in K_i(X')$.
Now we choose a nowhere dense closed subscheme $Y\hookrightarrow X$ such that $p$ is an isomorphism outside $Y$ and obtain an abstract blow-up square
\[
\xymatrix{
X' \ar[d] & Y' \ar[l] \ar[d]\phantom{.} \\
X & \ar[l] Y.
}
\]
From our main Theorem \ref{intro.main.thm}, we get a long exact sequence
\[
\cdots\to \prolim{n} K_{i+1}(Y'_n) \to K_i(X)\to  \prolim{n} K_i(Y_n)\oplus K_i(X') \to \cdots 
\]
of pro-groups.
Since $\operatorname{dim}(Y')<d$ and $\operatorname{dim}(Y)<d$, the pro-groups from the sequence involving $K_{i+1}(Y'_n)$ and $K_i(Y_n)$ vanish by the induction hypothesis on $d$.
Hence $p^ *\colon K_i(X)\to K_i(X')$ is injective and consequently $\xi=0$.
\end{proof}

We next show  that a noetherian scheme of dimension $d$ is $K_{-d}$-regular:

\begin{proof}[Proof of Theorem~\ref{intro.thm.wc}(ii)]
Again, we argue inductively on the dimension $d$ of the noetherian scheme $X$.
In every step of the induction, we deal with each $r\geq 0$ separately and show that $f^*\colon K_i(X)\to K_i(\mathbb{A}^r_X)$ is an isomorphism for $i\leq d$.
Since the projection $f\colon \mathbb{A}^r_X\to X$ has a section, the induced morphism $f^*$ is always a monomorphism.
In fact, we even have a decomposition $K(\mathbb{A}^r_X)\simeq K(X)\oplus N^{(r)}K(X)$ of spectra.
By applying Proposition~\ref{prop:reduceweibeltolocal} to the presheaf $N^{(r)}K[1]$ of spectra we may assume in each step that $X$ is noetherian and affine.
As non-positive algebraic $K$-theory of affines is nil-invariant, we can also take $X$ to be reduced.

It remains to show surjectivity of $f^*$.
The case $d=0$ is well-known \cite[Thm.~II.7.8]{Weib13}.
Let $d>0$ and assume the statement for all noetherian schemes of dimension less than $d$.

Let $i\leq -d$, $r\geq 0$ and consider an element $\xi\in K_i(\mathbb{A}^r_X)$.
We show that $\xi$ lies in the image of $f^*$. 
By applying \cite[Prop.~5]{KerzStrunk16} to $f$, we find a projective birational morphism $p\colon X'\to X$ such that $\tilde p^*(\xi)=0\in K_i(\mathbb{A}^r_{X'})$.
We choose again a nowhere dense closed subscheme $Y\hookrightarrow X$ such that $p$ is an isomorphism outside $Y$.
Theorem \ref{intro.main.thm} shows that the horizontal sequences of the diagram of pro-groups
\[
\xymatrix@C=4ex{
 \prolim{n} K_{i+1}(Y'_n)\ar[r]\ar[d]^{f^{*}_{Y'}} & K_i(X) \ar[r]\ar[d]^{f^*}&  \prolim{n} K_i(Y_n)\oplus K_i(X')  \ar[d]^{f^{*}_{Y}\oplus f^{*}_{X'}}\\
 \prolim{n} K_{i+1}(\mathbb{A}^r_{Y'_n})\ar[r] & K_i(\mathbb{A}^r_{X}) \ar[r]&  \prolim{n} K_i(\mathbb{A}^r_{Y_n})\oplus K_i(\mathbb{A}^r_{X'})&
}
\]
are exact.
Since $\operatorname{dim}(Y')<d$ and $\operatorname{dim}(Y)<d$, the vertical maps $f^{*}_{Y'}$ and $f^{*}_{Y}$ are isomorphisms by the induction hypothesis.
The pro-group in the upper horizontal sequence involving $K_i(Y_n)$ vanishes by part {\rm(i)} of Theorem~\ref{intro.thm.wc}.
Now a simple diagram chase shows that $\xi$ lies in the image of $f^*$.
\end{proof}

\subsection{Cdh-descent for homotopy \texorpdfstring{$K$}{K}-theory}\label{subsec:hok}

In this subsection we prove Theorem~\ref{intro.thm.kh}.
Let $X$ be a noetherian scheme of dimension $d<\infty$ and consider an abstract blow-up square~\eqref{intro.eq1}.
We study the bi\-relative spectra of $KH$-theory and $K$-theory of the square:
\begin{align*}
FH(n) &= KH(X,Y_n, \tilde X, E_n)\\
F(n)_\bullet &= K(X\times \Delta^\bullet,Y_n\times \Delta^\bullet , \tilde X \times \Delta^\bullet, E_n\times \Delta^\bullet).
\end{align*}
By definition \cite[IV.12]{Weib13} we have
\[
FH(n) = \colim_{\bullet\in\Delta^{op}}  F(n)_\bullet
\]
in the $\infty$-category of spectra. 

\begin{lem}\label{app.lemvan}\mbox{}
\begin{itemize}
\item[(i)]
 $FH(n+1)\xrightarrow{\simeq} FH(n)$ is an
equivalence for all $n>0$. 
\item[(ii)] 
For all $i<-d-2$, $n>0$ and $j\ge 0$ we have
$\pi_{i} (F(n)_j )= 0$.
\end{itemize}
\end{lem}

\begin{proof}
The assertion (i) follows directly from nil-invariance of $KH$-theory \cite[Cor.~IV.12.5]{Weib13}. For (ii) we use the fibration sequence
\begin{equation}\label{app.eq1}
F(n)_j \to  F'_j  \to F''(n)_j
\end{equation}
with 
\begin{align*}
F'_j &= \hofib( K(X\times \Delta^j) \to K(\tilde X\times \Delta^j)), \\
F''(n)_j &=  \hofib( K(Y_n\times \Delta^j) \to K(E_n\times \Delta^j) ).
\end{align*}
From the  exact homotopy sequences
\begin{align*}
&K_{i+1}(\tilde X\times \Delta^j) \to  \pi_i (F'_j) \to K_i(X\times \Delta^j), \\
&K_{i+2}(E_n\times \Delta^j) \to   \pi_{i+1} (F''(n)_j) \to K_{i+1}  K(Y_n\times \Delta^j) ,
\end{align*}
and Theorem~\ref{intro.thm.wc} we deduce that $\pi_i (F'_j)= \pi_{i+1} (F''(n)_j) =0$. The
long exact homotopy sequence corresponding to~\eqref{app.eq1} implies that  $\pi_{i} (F(n)_j )= 0$.
\end{proof}

\begin{proof}[Proof of Theorem \ref{intro.thm.kh}]
In order to prove Theorem~\ref{intro.thm.kh} we have to show that
\[
KH(X,Y, \tilde X, E)=FH(1)
\]
is contractible.
In view of Lemma~\ref{app.lemvan} the spectral sequence of the homotopy colimit 
\begin{equation}\label{app.eq22}
E^{2}_{pq}(n) = \pi_p ( \pi_q (F(n)_\bullet) ) \Rightarrow \pi_{p+q} KH(X,Y, \tilde X, E)
\end{equation}
 is convergent  (in fact uniformly convergent in $n$).

From Theorem~\ref{intro.main.thm} we know that 
\[
\prolim{n}  F(n)_j \simeq *
\]
for all $j\ge 0$. Considering the pro-system in $n$ of the spectral
sequence~\eqref{app.eq22} we deduce from
\[
\prolim{n} E^{2}_{pq}(n)  = 0 \quad \text{ for all } p,q\in \Z
\]
that the right-hand side of~\eqref{app.eq22} vanishes.
\end{proof}

\subsection{Identification with cdh-cohomology}\label{subsec.cdh}

In this subsection we prove Corollary~\ref{intro.cor}.
For a noetherian scheme $X$ of finite dimension let ${\rm Sch}_X$ be the category of schemes which are
separated and of finite type over $X$.
Let 
\[ \rL_\cdh\colon \mathrm{PSh}_\Sp({\rm Sch}_X) \to \mathrm{Sh}_\Sp({\rm Sch}_X^\cdh)\] be the
cdh-sheafification functor from the $\infty$-category of presheaves of spectra to the
$\infty$-category of sheaves of spectra. The functor $\rL_\cdh$ is left adjoint to the
inclusion functor of presheaves into sheaves, it preserves finite limits and small
colimits. Note that Theorem~\ref{intro.thm.kh} implies that $KH$ is an object of
$ {\rm Sh}_\Sp({\rm Sch}_X^\cdh)$. Let us denote by $\tilde K$
the connective  $K$-theory.

\begin{thm}\label{subsec.cdh.thm}
For a finite
dimensional noetherian scheme $X$ the canonical maps 
\[\rL_\cdh \tilde K \xrightarrow{\simeq} \rL_\cdh K \xrightarrow{\simeq} KH
\]
are equivalences of sheaves of spectra on ${\rm Sch}_X^\cdh$.
\end{thm}

In~\cite{Haes} Haesemeyer shows Theorem~\ref{subsec.cdh.thm} for varieties in
characteristic zero.
In particular, this theorem implies that there is a convergent spectral
sequence
\eq{subsec.cdh.eqspec}{
E^{p q}_2 = H^p_\cdh (X, \ra_\cdh \tilde K_{-q} ) \Rightarrow KH_{-p-q}(X)
}
where $\ra_\cdh$ denotes the sheafification functor on abelian sheaves.
Corollary~\ref{intro.cor} is a consequence of the spectral
sequence~\eqref{subsec.cdh.eqspec} in view of the following facts:
\begin{itemize}
\item[(i)]
$E^{p q}_2 = 0 $ for $p>d=\dim(X)$ by~\cite[Sec.~12]{SV},
\item[(ii)]  $\ra_\cdh \tilde K_0\cong\Z$ as this is already true for the Zariski topology,
\item[(iii)] $K_{-d}(X)\cong KH_{-d}(X) $ by Theorem~\ref{intro.thm.wc} and the spectral
  sequence relating $K$ and $KH$~\cite[Thm.~IV.12.3]{Weib13}.
\end{itemize}

In the proof of Theorem~\ref{subsec.cdh.thm} we need the following proposition, which is
well-known in case $X$ can be desingularized, see~\cite[Thm.~V.6.3]{Weib13}. The analog of
Proposition~\ref{subsec.cdh.prop2} for negative $K$-theory was shown
in~\cite[Prop.~5]{KerzStrunk16}. Our approach is to use devissage for $K$-theory of a
certain {\em Zariski--Riemann space} and {\em platification par \'eclatement} instead of
resolution of singularities.
Recall that for a scheme $X$ and an integer $i$ the group $NK_{i}(X)$ is defined to be the cokernel of the split injection $K_{i}(X) \to K_{i}(\A^{1}_{X})$.
\begin{prop}\label{subsec.cdh.prop2}
Let $X$ be a reduced scheme which is quasi-projective over a noetherian ring. Let $f\colon Y\to
X$ be a smooth and quasi-projective morphism. For any $\xi\in NK_i(Y)$, $i\in \Z$, there exists a projective birational morphism $p\colon X'\to X$ such that $\tilde
p^*(\xi)=0 \in NK_i(Y')$, where $ \tilde p\colon Y'\to Y$ is the base change of $p$.
\end{prop}

Before we begin the proof we introduce some notation. Let $\frY$ be a locally
ringed space with coherent structure sheaf, see~\cite[Sec.\ 0.5.3]{EGA} and~\cite[Sec.~0.4.1]{FK}.
We denote by $\Coh(\frY)$ the abelian category of $\sO_\frY$-modules which are of finite
presentation. By $\Vect(\frY)$ we denote the full exact subcategory of $\Coh(\frY)$ with objects
the locally free sheaves.

  We denote by $\Nil^+(\frY)$ the abelian category with objects
$(\sV,\nu)$ where
\begin{itemize}
\item $\sV$ is an $\sO_\frY$-module  of finite presentation, 
\item $\nu\colon\sV\to
\sV$ is an $\sO_\frY$-linear morphism such that $\nu^k=0$ for some $k>0$. 
\end{itemize}
The morphisms in $\Nil^+(\frY)$
are the $\sO_\frY$-linear maps compatible with the $\nu$-action.
By $\Nil(\frY)$ we denote the full exact subcategory of $\Nil^+(\frY)$ of objects
$(\sV,\nu)$ such that $\sV$ is locally free. 
\begin{lem}\label{subsec.cdh.lemplat}
Consider $f\colon Y\to X$ as in Proposition~\ref{subsec.cdh.prop2} and $\sU,\sV\in \Coh(Y) $ respectively
$\mathbb V=(\sV,\nu)\in \Nil^+(Y) $.  Then with the notation used there the following assertions are true.
\begin{itemize}
\item[(i)] If $\sV$ has Tor-dimension $\le 1$ over $X$ and $p\colon X'\to X$ is projective and
  birational with $X'$ reduced then $\tilde p^*(\sV)\cong L\tilde p^*(\sV)$. In particular on
  the full subcategory of $\Coh(Y)$ of sheaves of Tor-dimension $\le 1$ over $X$ the
  pullback $\tilde p^*$ is exact.
\item[(ii)]
Let $\phi\colon \sU\to \sV$ be a morphism and assume that $\sU$, $\sV$ and $\coker
(\phi)$ have Tor-dimension $\le 1 $ over $X$. Let $p\colon X'\to X$ be a projective and
  birational morphism with $X'$ reduced. Then the canonical maps 
\[
\tilde  p^*(\ker(\phi))\xrightarrow{\cong} \ker (\tilde p^*(\phi)) \quad\text{and}\quad
\tilde p^*(\im(\phi))\xrightarrow{\cong} \im (\tilde p^*(\phi))
\]
are isomorphisms.
\item[(iii)]
 There exists a projective birational morphism $p\colon X'\to X$ with $X'$ reduced such that the base change
$\tilde p^*(\sV)$ has Tor-dimension $\le 1$ over $X'$ and such that
there exists  a finite resolution
\[
0\to \mathbb V_l \to \cdots \to \mathbb V_0 \to \tilde p^* \mathbb  V \to 0 
\]
with $\mathbb V_i\in \Nil(Y')$ for all $0\le i\le l$.
\item[(iv)] Let $\phi\colon \sU\to \sV$ be a morphism in $\Coh(Y)$. There exists a projective
  birational morphism $p\colon X'\to X$ with $X'$ reduced such that  $\tilde p^* \sU $, $\tilde p^* \sV $,  $ \ker(\tilde p^*(\phi))$,
  $ \im(\tilde p^*(\phi))$ and  $ \coker(\tilde p^*(\phi))$ have Tor-dimension $\le 1$
  over $X'$.
\end{itemize}
\end{lem}

\begin{proof} 
For (i) choose an exact sequence of $\sO_Y$-modules
\[
0\to \sV_1\xrightarrow{\phi} \sV_0 \to \sV \to 0
\]
with $\sV_1$ and $\sV_0$ flat over $X$. Then 
\[
0\to\tilde p^* \sV_1\xrightarrow{\phi}\tilde p^*  \sV_0 \to\tilde p^*  \sV \to 0
\]
is exact, since it is so over the maximal points of $X'$.
(ii) is a direct consequence of (i).

For (iii) we argue by induction on $k>0$ with $\nu^k=0$ for all $Y\to X$ at once.
Choose a locally free presentation of $\sO_Y$-modules
\[
\sV_1\xrightarrow{\phi} \sV_0 \to \sV \to 0.
\]
For any projective birational morphism $p\colon X'\to X$ with $X'$ reduced the image
$\im(\tilde p^*(\phi))$ is the strict transform of the image $\im (\phi)$. So
by~\cite[Thm.~5.2.2]{GR} we can find a $p$ such that the former image is flat over $X'$.
Write $\mathbb V'=(\sV',\nu')$ for $\tilde p^*\mathbb V$. By the same technique we can
assume that additionally $\coker (\nu')^{k-1}$ has Tor-dimension $\le 1$ over $X'$ in case
$k>1$. This implies that $\ker (\nu')^{k-1}$ and $\im (\nu')^{k-1}$ have Tor-dimension
$\le 1$ over $X'$. So by (ii) the
formation of the image and the kernel of $(\nu')^{k-1}$ is
compatible with pullback along a further modification of $X'$.

Let $\mathbb W=(\ker(\nu')^{k-1} , \nu'|_{\ker(\nu')^{k-1}}) $. By our induction assumption in case
$k>1$  we can replace $X'$ by a further modification such  that afterwards there exists a finite resolution $\mathbb W_*\to
\mathbb W$ with $\mathbb W_i \in \Nil(Y')$ for all $i\ge 0$. By the analog
of~\cite[Lem.~6]{KerzStrunk16} we observe that  $\im (\nu')^{k-1}$ has finite
Tor-dimension over $Y'$. So standard techniques allow us to construct a finite, locally
free resolution of  $\im (\nu')^{k-1}$ which can be spliced to $\mathbb W_*$ to provide the
requested resolution $\mathbb V_*$ of $\mathbb V'$.

The proof of (iv) also relies on~\cite[Thm.~5.2.2]{GR} and is similar to  the proof of (iii).
\end{proof}

\begin{proof}[Proof of Proposition~\ref{subsec.cdh.prop2}]
Let $\mathfrak I$ be the cofiltered category of reduced schemes which are projective and
birational over $X$. We consider the Zariski--Rie\-mann type locally ringed space
\[
 \frY= \lim_{X'\in \mathfrak I} Y\times_X X'.
\]
The limit is taken in the category of locally ringed topological spaces. We use standard 
properties of sheaves of $\sO_\frY$-modules as presented in~\cite[Sec.~0.4.2]{FK} without
further explanation. Beside those we also use the following properties:

\begin{itemize}
\item[(i)]  $\sO_\frY$ is coherent.
\item[(ii)] Any $\mathbb V\in \Nil^+(\frY)$ has a finite resolution $\mathbb V_*\to \mathbb
  V$ with
  $\mathbb V_i\in \Nil(\frY)$ for all $i\ge 0$.
\end{itemize}

In order to prove (i) consider an $\sO_\frU$-linear map $\mathfrak f \colon \sO_\frU^n \to\sO_\frU$
for some open subspace $\frU$ of $\frY$. We have to show that $\ker(\mathfrak f)$ is of finite
type. Without loss of generality $\frU=\frY$ and $\mathfrak f$ is induced by a morphism
$\phi\colon\sO_{Y}^n\to \sO_Y$. By Lemma~\ref{subsec.cdh.lemplat}(iv) we find a projective
birational morphism $p\colon X'\to X$ with $X'$ reduced such that $\coker(\tilde p^*(\phi))$ has Tor-dimension $\leq 1$
 over $X'$. After shrinking $Y'$ around a given point there exists a surjection $\psi\colon \sO_{Y'}^m\to \ker(\tilde p^*(\phi))$. Now
Lemma~\ref{subsec.cdh.lemplat}(ii) implies that the pullback of $\psi$ to $\frY$ induces a
surjection $\sO_{\frY}^m\to \ker(\mathfrak f) $. 

Using a similar argument (ii) can be deduced from parts (i) and (iii) of Lemma~\ref{subsec.cdh.lemplat}.

We now study $K$-theory of $\frY$.
For any $Y'=Y\times_X X'$ with $X'\in \mathfrak I$ 
\[ 
NK_{i+1}(Y')\cong \coker( K_i(\Vect(Y'))\to  K_i(\Nil(Y'))    )
\]
by \cite[Thm.~V.8.1]{Weib13}. As $K$-theory commutes with filtered colimits of exact
categories we get
\eq{subsec.cdh.eq61}{
K_i(\Vect(\frY))\cong \colim_{Y'} K_i(\Vect(Y'))  \text{ and } K_i(\Nil(\frY))\cong \colim_{Y'}  K_i(\Nil(Y')). 
}
In order to prove Proposition~\ref{subsec.cdh.prop2} we have to show that the left group
surjects onto the right group in~\eqref{subsec.cdh.eq61}.
From (ii) and the resolution theorem~\cite[Thm.~V.3.1]{Weib13}
we deduce that 
\[
K_i(\Vect(\frY)) \cong K_i(\Coh(\frY))  \quad\text{and}\quad K_i(\Nil(\frY)) \cong K_i(\Nil^+(\frY)).
\]
Devissage~\cite[Thm.~V.4.1]{Weib13} tells us that 
\[
K_i(\Coh(\frY)) \cong  K_i(\Nil^+(\frY)).
\]
These results together finish the proof of Proposition~\ref{subsec.cdh.prop2}.
\end{proof}

\begin{proof}[Proof of Theorem~\ref{subsec.cdh.thm}]
The ordinary sheafification functor $\ra_\cdh$ from pre\-sheaves of sets to sheaves of sets on
${\rm Sch}_X^\cdh$ preserves finite limits and small colimits. As the $\infty$-topos ${\rm Sh}({\rm Sch}_X^\cdh)$ is
hypercomplete we have to show that 
\begin{align}\label{subsec.cdh.eq01}
\ra_\cdh K_i &= 0  &\text{for } i<0,  \\ \label{subsec.cdh.eq02}
\ra_\cdh K_i &\xrightarrow{\cong} \ra_\cdh KH_i  &\text{for } i\in \Z.  
\end{align}
In order to verify~\eqref{subsec.cdh.eq01} we show by induction on $\dim(Y)$ that the map
\begin{equation}\label{subsec.cdh.eq03}
K_i(Y) \to \ra_\cdh K_i (Y)
\end{equation}
vanishes for any affine $Y\in {\rm Sch}_X$ and $i<0$. Without loss of generality $Y$ is reduced. Consider
$\xi \in K_i(Y)$. By~\cite[Prop.~5]{KerzStrunk16} there is a projective birational
morphism $\phi\colon Y'\to Y$ such that $\phi^*(\xi)=0$. Let $Z\to Y$ be a closed subscheme such that
$\phi$ is an isomorphism over $Y\setminus Z$ and such that $\dim(Z)<\dim(Y)$.
Then
the image of $\xi$ in $\ra_\cdh K_i(Z)$ vanishes by our induction assumption. As $Y'\coprod Z \to Y$ is a cdh-covering we deduce
that $\xi$ lies in the kernel of~\eqref{subsec.cdh.eq03}.

In order to verify~\eqref{subsec.cdh.eq02} we apply $\ra_\cdh$ to the spectral
sequence~\cite[Thm.~IV.12.3]{Weib13} relating $K$ and $KH$. The weak convergence of this
spectral sequence involves only kernels, cokernels and small filtered colimits, so we get a weakly convergent spectral sequence
\[
E^1_{p q} = \ra_\cdh N^p K_q  \Rightarrow \ra_\cdh KH_{p+q}.
\]
in the category of abelian sheaves on ${\rm Sch}_X^\cdh$.

In analogy with the proof of~\eqref{subsec.cdh.eq01} we deduce from Proposition~\ref{subsec.cdh.prop2}  that $E^1_{p q}=0$ for all $p>0$, $q\in \Z$.
\end{proof}

\appendix

\section{Homology of monoids}\label{ex.sec.hom}

In this appendix we present some background material on the homology of simplicial
monoids, which plays an important role in the proof of
Proposition~\ref{ex.prop.trivac}.
  For a discrete  monoid $G$ and a set $X$ on which $G$ acts let
$\mathbf E_X G$ be the action category of $G$ on $X$, i.e.\ the objects are the elements
of $X$ and 
\[\Hom_{\mathbf E_X G}(x,y)= \{g\in G \, |\, g(x)=y  \}.  
\]
Let the simplicial set $\rE_X G$ be the nerve of $\mathbf E_X G$. For a
simplicial monoid $G$ acting on the simplicial set $X$ we can perform this construction
degreewise and take the diagonal to get a simplicial set $\rE_X G$. We write $\rE G= \rE_G
G$. Note that $\rE_{\{*\}}G = \rB G$. Since in the discrete case $\mathbf E_G G$ has an initial object, $\rE
G$ is contractible for any simplicial monoid $G$.

Concretely, $\rE G$ is given by
\[
[n] \mapsto G^{n+1} , \quad
d_i(g_0 , \ldots ,g_n ) =  \begin{cases} 
   ( g_0,  \ldots , g_{n-i-1}g_{n-i}, \dots, g_{n} )  & \text{if } i < n \\
   (g_1, \ldots ,   g_n)      & \text{if } i=n.
  \end{cases} 
\]

 For a simplicial $G$-module $M$ we write $C(G;M)$ for
the simplicial abelian group $\Z \rE G\otimes_{\Z G} M $  and $H_i(G;M)$ for $\pi_i
C(G;M)$. If $M=\Z$ we write $C(G)$ for $C(G;\Z )$ and $H_i(G)$ for $H_i(G;\Z)$.

Note that if $G$ and $M$ are discrete
\[
H_i(G; M) = {\rm Tor}^{\Z G}_i ( \Z , M ),
\]
as $\Z \rE G$ is a $\Z G$-free resolution of $\Z$.
If $G_1$ and $G_2$ are simplicial monoids there is a canonical isomorphism
\begin{equation}\label{ex.eqdecp}
C(G_1 \times G_2) \cong C(G_1) \otimes_\Z C(G_2).
\end{equation}

 In the proof of Proposition~\ref{ex.prop.trivac} we
use some  facts about the homology of a semi-direct product $G \ltimes H$, where $G$ is a
simplicial monoid acting on the simplicial group $H$ from the left. The
multiplication on $G \ltimes H$ is given by 
\[
(g,h) , (g',h') \mapsto (g g',h g(h')).
\]
Note that the relative homology $H_i(G \ltimes H, G)$ is given by the cokernel of the (split) injection
\[
 H_i(G) \to H_i(G \ltimes H) .
\]

\begin{prop}\label{ex.prop.hosemdi}
Assume $G$ is a
group-like simplicial monoid acting on the simplicial group  $H$. 
Then we get a canonical
isomorphism 
\begin{equation}\label{ex.semproco}
 H_i(G; C(H))\xrightarrow{\cong}   H_i(G \ltimes H) .
\end{equation}
\end{prop}

\begin{proof}
In this proof we treat $\rE_{\rB H} G $ and $\rB (G\ltimes H)$ as bisimplicial sets.
We write $(\rE_{\rB H} G)_n$ for the simplicial set obtained by applying the $n$-th simplicial degree of the $\rE$-construction to the simplicial monoid $G$ and the simplicial $G$-set $\rB H$, and similarly for $\rB( G \ltimes H)$.
We consider the  morphism of bisimplicial sets
\begin{equation}
\phi\colon \rE_{\rB H} G  \xrightarrow{\simeq} \rB (G\ltimes H)
\end{equation}
defined as follows.
Consider $w=(g_0,\ldots , g_{n-1}, h_0, \ldots , h_{n-1} )\in (\rE_{\rB H} G)_n$ and let the
$i$-th component of $\phi(w)$ be $(g_i,g_i  \cdots g_{n-1}(h_i))$.
Note that
\[
 H_i(G; C(H)) = \pi_i \Z  \rE_{\rB H} G \quad \text{ and } \quad H_i(G \ltimes H)=\pi_i\Z \rB (G\ltimes H).
\]

We claim that for each $n\ge 0 $ the map $\phi_n \colon (\rE_{\rB H} G)_n \to \rB(G\ltimes
H)_n$ is an equivalence of simplicial sets, so that the diagonal of $\phi$ is also an
equivalence  \cite[Prop.~IV.1.9]{GJ}.
As $\phi_n$ is built out of the map $G\ltimes H \to G\ltimes H$ with
$(g,h)\mapsto (g,g(h))$, which is an equivalence by Lemma~\ref{ex.eqlemsimp} below, the claim follows.
\end{proof}

\begin{lem}\label{ex.eqlemsimp}
For a group-like simplicial monoid $G$  acting on the simplicial group  $H$ from the left
the map
\[
\psi\colon  G\ltimes H \to G\ltimes H, \quad (g,h)\mapsto (g,g(h))
\]
is an equivalence.
\end{lem}

\begin{proof} 
The projection $\pi\colon G\ltimes H \to G$ is a fibration \cite[Cor.~V.2.6]{GJ}, so by
\cite[Lem.~IV.5.2]{GJ}, it suffices to show that for each $0$-simplex
$\sigma\colon \Delta^0\to G$ the map $\psi$ induces an equivalence on the fibre $\psi_\sigma\colon   (G\ltimes
H)_\sigma \to (G\ltimes H)_\sigma$.  Note that $\psi_\sigma $ is isomorphic to the action
of $\sigma$ on $H$, which is clearly homotopic to the identity if $\sigma$ is in the
connected component of the identity of $G$. But as $\pi_0(G)$ is a group we can reduce to the
latter case.
\end{proof}

\begin{defn}\label{ex.def.homoto}
  Let $C_1$ and $C_2$ be simplicial abelian groups with an action of a simplicial monoid
  $G$, let $m\ge 0$. An \emph{$m$-homotopy} between two $G$-equivariant morphisms
  $\phi_0, \phi_1\colon C_1\to C_2$ is a $G$-equivariant morphism of simplicial abelian groups
  $h\colon C_1\otimes_\Z \Z \Delta^1 \to C_2$ with $\phi_0=h_{C_1\times \{0 \} }$ and
  $\phi_1=h_{C_1\times \{ 1 \} }$ in simplicial degrees $\le m$ .

We say that a $G$-equivariant morphism of simplicial groups  $\phi\colon H_1\to H_2$ is \emph{homologically $m$-constant} if \[ C(\phi)\colon C(H_1)\to C(H_2) \] is  $G$-equivariantly
$m$-homotopic
in the above sense to $C(e)$, where $e\colon H_1\to H_2$ is the constant morphism.
\end{defn}

As an immediate consequence of Proposition~\ref{ex.prop.hosemdi} we obtain:

\begin{lem}\label{lem.vanreho}
Let $G$ be a group-like simplicial monoid.
If $\phi\colon H_1\to H_2$  is a $G$-equivariant morphism of simplicial 
  groups which  is homologically $m$-constant, then the map
\[
H_i(G\ltimes H_1 , G) \xrightarrow{\phi_*} H_i(G\ltimes H_2 , G)
\]
vanishes for $i\le m$.
\end{lem}

The following observations about group homology are classical, but as we need functorial
homotopies of simplicial abelian groups,  we  recall them briefly.

\begin{const}[Eilenberg--Zilber]\label{const.eilzil} \mbox{}\\
Suppose we have given the following data.
\begin{itemize}
\item
 $T$ and $H$  discrete groups.
\item
 $\phi,\phi'\colon T\to H$ group
homomorphisms with commuting images. 
\item
 An $m$-homotopy between $C(\phi)$ and $C(e)$ and an $m'$-homotopy between $C(\phi')$ and
 $C(e)$, where $e$ denotes the constant homomorphism and  $m,m'\ge 0$.
\end{itemize}

\noindent
Then there exists an $(m+m'+1)$-homotopy between
\[
C(\phi\cdot \phi') \quad\text{and}\quad C(\phi) +C( \phi')\colon  C(T) \to C(H)  . 
 \]
\end{const}

Construction~\ref{const.eilzil} is canonical and functorial in the given data.
In order to show the existence of the desired homotopy, we assume for simplicity of
notation that $H$ is abelian. It is  sufficient to construct an $(m+m'+1)$-homotopy between 
\[
C(\phi\times \phi') \quad \text{and}\quad C(\phi\times e) +C(e\times \phi')\colon 
C(T\times T) \to C(H\times H)
 \]
by composition with  multiplication and diagonal. 
In view of~\eqref{ex.eqdecp} and as $C(\phi)$ respectively $C(\phi')$ are canonically homotopic
(depending on the input data) to  maps which are equal to $C(e)$ in degrees $\le m$ respectively 
$\le m'$ it is sufficient to observe the following fact.

Writing $C=C(T)$, $D=C( H)$, $\varphi=C(\phi)$ and $\varphi'=C(\phi')$ the induced maps of
normalized complexes $N \varphi\colon NC\to ND$ and $N \varphi'\colon NC\to ND$ vanish
in degrees $[1,m]$ respectively $[1,m']$. We claim that the morphisms
\[
\varphi \otimes \varphi' \quad\text{and}\quad \varphi\otimes e +e\otimes \varphi'\colon
C\otimes_\Z C \to D\otimes_\Z D  
 \]
are canonically $(m+m'+1)$-homotopic. The requested $(m+m'+1)$-homotopy for the associated
normalized complexes is defined as the canonical homotopy
\begin{align*}
N (\varphi \otimes \varphi')& \simeq  \nabla \circ ( N \varphi \otimes N\varphi') \circ
{\rm AW}  \\
& = \nabla \circ (   N \varphi\otimes e +e\otimes N\varphi' )\circ {\rm AW}  \quad
  \text{(in  ${\rm deg}\le m+m'+1$)}
 \\
& \simeq   N (\varphi\otimes e) +N( e \otimes \varphi'). 
\end{align*}
Here $\nabla$ is the shuffle map and $\rm AW$ is the Alexander--Whitney map,
see~\cite[2.3]{SchwedeShipley03} for an overview and \cite[Sec.~2]{EM} for a detailed
construction of these canonical homotopies.

\begin{const}[Conjugation invariance]\label{ex.const.conj} \mbox{}\\
Suppose we have given the following data.
\begin{itemize}
\item A discrete group $H$.
\item An element $u\in H$ with an associated conjugation automorphism $\phi_u\colon H \to H$, $h\mapsto u^{-1} h  u$.
\end{itemize}

\noindent
Then there exists a homotopy between
\[
C({\rm id}_H) \quad\text{and}\quad C( \phi_u )\colon C(H) \to C(H).
\]
\end{const} 

Construction~\ref{ex.const.conj} is canonical and functorial in the pair $(H,u)$. It is
sufficient to observe that $u^{-1}$ represents a natural
transformation between the endo-functors ${\rm id}_H$ and $\phi_u$  of the groupoid
$\mathbf E_{\{ * \}}H$. This natural transformation gives rise to a homotopy of self-maps of the
classifying space $\rB H$. The induced homotopy on $C(H)=\Z \rB H$ is the requested one.

\bibliographystyle{amsalpha}
\providecommand{\bysame}{\leavevmode\hbox to3em{\hrulefill}\thinspace}
\providecommand{\href}[2]{#2}

\end{document}